\documentclass[11pt,a4paper,reqno]{amsart}
\usepackage{amssymb, amsmath, latexsym, mathrsfs, graphicx, stmaryrd,
amsthm, psfrag, color, multirow, amsthm}
\usepackage{rotating}
\usepackage{url}
\usepackage[enableskew, vcentermath]{youngtab}
\usepackage[margin=1.3in]{geometry}

\usepackage{amssymb}
\usepackage{mathrsfs}
\usepackage{mathtools}
\usepackage{amsbsy}
\usepackage[all]{xy}
\usepackage{bm}
\usepackage[backref=page,linktocpage]{hyperref}
\usepackage{tikz}
\usepackage{array}
\usepackage{float}
\usepackage{enumerate}
\usepackage{xcolor}
\usepackage{hhline}
\allowdisplaybreaks
\usepackage[noadjust]{cite}
\usepackage{algorithm}
\usepackage[noend]{algpseudocode}
\usepackage[section]{placeins}

\usepackage{caption}
\usepackage{subcaption}
\usepackage{tabu}
\usepackage{diagbox}
\usepackage{bbm}
\usepackage{booktabs}

\usepackage[noabbrev,capitalise]{cleveref}

\newenvironment{enumerate*}%
  {\begin{enumerate}[(I)]%
    \setlength{\itemsep}{10pt}%
    \setlength{\parskip}{0pt}}%
  {\end{enumerate}}

\newtheorem{theorem}{Theorem}[section]
\newtheorem{proposition}[theorem]{Proposition}
\newtheorem{corollary}[theorem]{Corollary}
\newtheorem{conjecture}[theorem]{Conjecture}
\newtheorem{question}[theorem]{Question}

\newtheorem{lemma}[theorem]{Lemma}

\theoremstyle{definition}
\newtheorem{definition}[theorem]{Definition}
\newtheorem{remark}[theorem]{Remark}
\newtheorem{example}[theorem]{Example}

\DeclarePairedDelimiter{\abs}{\lvert}{\rvert}

\DeclareMathOperator{\conv}{conv}
\DeclareMathOperator{\relint}{relint}

\DeclareMathOperator{\Aut}{Aut}

\DeclareMathOperator{\Vol}{Vol}

\newcommand{\cB}{\mathcal B}

\newcommand{\cF}{\mathcal F}

\newcommand{\RR}{\mathbb R}
\newcommand{\ZZ}{\mathbb Z}

\usepackage[dvipsnames]{xcolor}

\hypersetup{
    colorlinks = true,
    linkbordercolor = {white},
    linkcolor = {NavyBlue},
    anchorcolor = {black},
    citecolor = {NavyBlue},
    filecolor = {cyan},
    menucolor = {NavyBlue},
    runcolor = {cyan},
    urlcolor = {NavyBlue}
}

\title[Matroid toric ideals without Gr\"obner bases]{There are matroid toric ideals\\without  quadratic Gr\"obner Bases}

\author{Jes\'us A. De Loera}
\author{Luis Ferroni}
\author{Santiago Morales}
\author{J\"org Rambau}
\date{\today}

\address{(J. De Loera \& S. Morales) University of California, Davis, United States}
\email{deloera@math.ucdavis.edu}
\email{moralesduarte@ucdavis.edu}

\address{(L. Ferroni) Universit\`a di Pisa, Italy}
\email{luis.ferroni@unipi.it}

\address{(J. Rambau) Universit\"at Bayreuth, Germany}
\email{joerg.rambau@uni-bayreuth.de}

\begin{document}

\begin{abstract}
    Our paper shows that if a matroid contains the Fano plane or its dual as a minor, then its toric ideal does not have any quadratic Gr\"obner basis. More than 25 years ago, Hibi, Herzog, and Sturmfels established a direct connection between the existence of quadratic Gr\"obner bases and regular unimodular flag triangulations. Our paper solves a famous question posed by Herzog and Hibi on a polyhedral reformulation for the existence of quadratic Gr\"obner bases: we show that the base polytopes of the Fano plane and its dual do not have regular unimodular flag triangulations which implies the main result on Gr\"obner bases. 
    Our proof relies on several novel tools: a lemma that connects the $1$-skeleton of a lattice polytope to the lattice points in its dilations, an encoding with Boolean formulas and SAT solvers, and symmetry-breaking arguments. 
\end{abstract}

\maketitle

\setcounter{tocdepth}{1}
\tableofcontents

\section{Introduction}

In 1980, Neil White conjectured that the toric ideal of a matroid base polytope is quadratically generated \cite{White1980}. This old
conjecture creates a liaison between matroids and commutative algebra and has attracted the attention of many mathematicians. A stronger version of White's conjecture was proposed by Herzog and Hibi in 2002, who asked whether every integer polymatroid toric ideal is $G$-quadratic, i.e.,  admits a quadratic Gröbner basis \cite[p.~241]{HerzogHibi2002}. Although this stronger property has been supported by several positive results (for example in \cite{Sturmfels1996, Blum2001, LamPostnikov2007, LamPostnikov2018, LamPostnikovPolypositroids}), not only this one, but also White's original conjecture remains unsolved after more than four decades. 

By work of Sturmfels \cite{Sturmfels1996}, together with an observation of Ohsugi and Hibi \cite{OhsugiHibi1999}, the existence of a quadratic Gröbner basis is equivalent to the existence of a \emph{regular unimodular flag\footnote{Some authors use the name \emph{quadratic triangulation} when alluding to regular unimodular flag triangulations.} triangulation} of the matroid base polytope. Therefore, the algebraic question posed by Herzog and Hibi can be recast into a problem in discrete geometry and polyhedra.

Until now, Herzog and Hibi's question was verified to have a positive answer for all except two matroids on a ground set with at most seven elements; see the work of Hayase, Hibi, Katsuno, and Shibata \cite{Hayase2022}. The two missing cases, both of which were analyzed by computer for over a year by the aforementioned authors without a conclusive answer, are the (in)famous \emph{Fano matroid $F_7$} and its dual, the matroid $F_7^*$. In other words, the existence of regular unimodular \emph{flag} triangulation for the base polytope of the Fano matroid \(P_{F_7}\) and its dual remains still open, and a positive answer would settle a (very strong) form of White’s conjecture for all matroids on at most seven elements. Conversely, if one were able to disprove the existence of regular unimodular flag triangulations for the base polytope of the Fano matroid (or, equivalently, its dual), this would settle in the negative the question raised by Herzog and Hibi.

Recent progress isolates the remaining difficulty to the \emph{flag} condition. Haws in his Ph.D. thesis conjectured that every matroid base polytope admits a unimodular triangulation \cite[Conjecture~57]{Haws2009}. This has now been proved true by Backman and Liu: every matroid base polytope admits a \emph{unimodular} triangulation \cite{LiuBackman2025} which can in fact be chosen to be regular. The Fano case therefore motivates the remaining question: can one enforce flagness while retaining both regularity and unimodularity?

The main result of this paper answers  the preceding question in the negative, and therefore also settles the question raised by Herzog and Hibi.

We remark that the obstruction found here is not regularity alone, nor unimodularity alone, nor flagness alone. As mentioned above, Backman and Liu showed that for every matroid $M$, the base polytope $P_M$ has a
regular unimodular triangulation \cite{LiuBackman2025}. On the other hand, we exhibit unimodular flag triangulations of \(P_{F_7}\). Although we found unimodular flag triangulations, what fails for \(F_7\) is the simultaneous satisfaction of the three requirements:
\[
\text{regular} \quad+\quad \text{unimodular} \quad+\quad \text{flag}.
\]

\smallskip
\begin{theorem}
    The base polytope of the matroid $F_7$ has no regular unimodular flag triangulations. 
\end{theorem}

\smallskip

The key methods and new ideas that made the solution of this old problem possible are the following. In any unimodular triangulation of a lattice polytope \(P\), every lattice point \(y\in tP\) has a unique carrier face with at most \(t\) vertices.  For a flag triangulation, carrier faces are cliques of the $1$-skeleton, so the necessary conditions for existence can be encoded in Boolean Conjunctive Normal Form (abbreviated CNF) using edge indicator variables. Under these constraints, we used SAT solvers to find unimodular flag triangulations of \(P_{F_7}\) (for more background on  Boolean SAT solvers, SAT-based mathematical proofs and symmetry breaking, see \cite{BiereHeuleVanMaarenWalsh2021, HeuleKullmann2017,
HeuleKullmannMarek2016, BrightKotsireasGanesh2022, SAT4Math}).  Independently, in an attempt to enumerate all unimodular flag triangulations of \(P_{F_7}\), we also found over $12{,}000$ (non-regular) unimodular flag triangulations by a newer version of TOPCOM that can restrict enumeration to unimodular (old feature) and flag (new feature, see Sec.~\ref{sec:TOPCOM}) triangulations during the process \cite{Rambau:TOPCOM-preprint:2026}.  The full enumeration, however, could not be completed before the publication of this paper.

For the Fano matroid case, the proof of our main result is based on a negative certificate of regularity. Regularity is imposed only through necessary height inequalities. In each
\emph{quadratic fiber} \(v_i+v_j=v_k+v_\ell\), the selected edge must be strictly lower than every other edge in the same fiber. We call this necessary condition quadratic coherence. For \(P_{F_7}\), the \emph{carrier constraints of level 2 and 3} (derived from the dilations $2P$ and $3P$), together with \emph{symmetry breaking} and exact Farkas-certified obstructions to quadratic coherence, give six UNSAT certificates. Together, these show the three properties cannot coexist for the Fano matroid.

As we demonstrate, the property of having regular unimodular flag triangulations is closed under minors and duals. In particular, the preceding statement shows the absence of regular unimodular flag triangulations for the base polytopes of all matroids containing the Fano matroid or its dual as minors. 

\begin{corollary}
    There exist infinitely many matroids whose toric ideals do not possess quadratic Gr\"obner bases.
\end{corollary}

The Fano matroid is a sparse paving matroid, and a well-known conjecture posed by Mayhew, Newman, Welsh and Whittle \cite[Conjecture~1.7]{Mayhew2011} postulates that for any fixed sparse paving matroid $N$, asymptotically almost all matroids must contain $N$ as a minor (the specific case $N = F_7$ is discussed in detail by van der Pol in \cite[Section~5]{vanderPol2023}). If that conjecture is true, then our result actually implies that asymptotically $0\%$ of matroids have a regular unimodular flag triangulation or, in other words, that asymptotically $0\%$ of matroid toric ideals have a quadratic Gr\"obner basis.

\smallskip
It is worth mentioning that our result does not disprove White's conjecture, i.e., we do not disprove that toric ideals of matroids are quadratically generated. Since the Fano matroid has rank \(3\), Kashiwabara's theorem implies that \(I_{F_7}\) is generated by quadrics \cite{Kashiwabara2010}. Thus \(F_7\) separates quadratic generation from the existence of a quadratic Gr\"obner basis.

\subsection*{Additional related work}

A second conjecture posed by White in \cite{White1980} states that matroid toric ideals are generated by quadratic polynomials that come from symmetric exchanges between pairs of bases. Over the years many authors have proved results that support this conjecture, which we will henceforth call White's \emph{strong} conjecture. Notably, Bonin \cite{bonin} proved that it holds for all sparse paving matroids, whereas Blasiak proved it for graphic and cographic matroids. Notably, Laso{\'n} and Micha{\l}ek \cite{lason-michalek} proved that it holds for all strongly base orderable matroids, in particular for all gammoids, and thus also for all transversal and all cotransversal matroids; moreover, these authors also showed that White's strong conjecture is true ``up to saturation'' for all matroids. Shibata \cite{Shibata} proved that both White's conjecture and White's strong conjecture are preserved under various matroid operations, including $2$-sums and series and parallel extensions. More recently, the strong version of White's conjecture was established for all paving matroids by Yu and Yuen in \cite{Yu-Yuen}, and it was also proved that the conjecture is preserved under circuit-hyperplane relaxations by Han, Micha{\l}ek and Weigert \cite{han-michalek-weigert}.  

\subsection*{Outline}

This paper is organized as follows. In Section \ref{basics} we review the necessary definitions and notation.  Section~\ref{sec:TOPCOM} is devoted to a method to enumerate all unimodular flag triangulations of a point configuration by  a new adaptation of symmetric lexicographic subset reverse search and our experience with it.  In Section \ref{idea-latticepts2triangs} we show our key idea that connects the lattice points inside dilations of the matroid polytope to the 1-skeleton of all unimodular triangulations. This allows us to setup and solve a SAT instance that was used to find the first counterexample.  In Section \ref{flaguni-implies-nonregular} we show that together flagness and unimodularity force the failure of regularity for the Fano matroid polytope, but in Section \ref{largermatroids} we discuss what happens with other matroids with more than seven elements and present some further observations. The Appendix contains an explicit example of a non-regular unimodular flag triangulation. 

\subsection*{Acknowledgements}

We are very grateful to Marijn Heule and Jos\'e Samper for inspiring exchanges. Selected calculations were performed using the \texttt{festus-cluster} of the Bayreuth Centre for High Performance Computing\footnote{\url{https://www.bzhpc.uni-bayreuth.de}}, funded by the Deutsche Forschungsgemeinschaft (DFG, German Research Foundation) -- 523317330.
The first and third author were partially supported by NSF grants 
DMS-2348578 and DMS-2434665. The second author is a member of the GNSAGA group of the Istituto Nazionale di Alta Matematica (INdAM).

\section{Preliminaries} \label{basics}

Throughout this paper we shall assume that the reader is acquainted with the basic terminology and notation used in matroid theory and polyhedral geometry. For background about matroids we refer to \cite{Welsh1976,Oxley,Ardila2018}, and for further details about polyhedra and triangulations we suggest
\cite{Ziegler,Sturmfels1996,DeLoeraRambauSantos}.

\begin{definition}[Matroid base polytope \cite{Edmonds2003}]
A \emph{matroid} is a pair $M=(E,\cB)$ where $E$ is a finite set and $\cB$ is a non-empty family of subsets of $E$, called bases, satisfying the basis-exchange axiom: for any $B_1,B_2\in\cB$ and $e\in B_1\setminus B_2$, there exists $f\in B_2\setminus B_1$ such that
\[
 (B_1\setminus\{e\})\cup\{f\}\in\cB.
\]
For $S\subseteq E$, let $\chi_S\in\{0,1\}^E$ be its indicator vector.  The \emph{base polytope} of $M$ is
\[
 P_M=\conv\{\chi_B:B\in\cB\}\subseteq\RR^E.
\]
\end{definition}

\begin{example}
    The \emph{Fano matroid} $F_7$ has ground set $E=\{0,1,\ldots,6\}$ and rank $3$.
    Its non-bases are the seven lines
    \[
    012,\quad 034,\quad 056,\quad 135,\quad 146,\quad 236,\quad 245.
    \]
    Equivalently, its bases are all $3$-subsets of $E$ except these seven triples.
    Thus $P_{F_7}$ has $28$ vertices. All vertices lie in the affine hyperplane
    \[
        x_0+x_1+\cdots+x_6=3,
    \]
    and since $F_7$ is connected, $\dim P_{F_7}=6$.

    The \emph{non-Fano matroid} $F_7^{-}$ is obtained from $F_7$ by relaxing any one
    of its seven lines. With the above labeling, let us relax the line $012$.
    Then $F_7^{-}$ has the same ground set and rank, and its non-bases are
    \[
    034,\quad 056,\quad 135,\quad 146,\quad 236,\quad 245.
    \]
    Equivalently, its bases are all $3$-subsets of $E$ except these six triples,
    so $P_{F_7^-}$ has $29$ vertices. Again all vertices lie in the hyperplane
    \[
        x_0+x_1+\cdots+x_6=3,
    \]
    and since $F_7^{-}$ is connected, $\dim P_{F_7^-}=6$.

    The following incidence diagrams make the relation between them clear:
    $F_7^{-}$ is obtained from $F_7$ by removing the circular line through $0,1,2$.

    \begin{center}
    \begin{minipage}{0.45\textwidth}
        \centering
        \begin{tikzpicture}[scale=0.9]
            \tikzstyle{edges}=[thick];
            \coordinate (p3) at (0,2.6);
            \coordinate (p4) at (-2.25,-1.3);
            \coordinate (p5) at (2.25,-1.3);
            \coordinate (p0) at (-1.125,0.65);
            \coordinate (p1) at (1.125,0.65);
            \coordinate (p2) at (0,-1.3);
            \coordinate (p6) at (0,0);

            \draw[thick] (p3)--(p4); 
            \draw[thick] (p3)--(p5); 
            \draw[thick] (p4)--(p5); 
            \draw[thick] (p3)--(p2); 
            \draw[thick] (p4)--(p1); 
            \draw[thick] (p5)--(p0); 
            \draw[thick] (p6) circle[radius=1.3]; 

            \fill (p0) circle (2pt); \node[left]  at (p0) {$0$};
        \fill (p1) circle (2pt); \node[right] at (p1) {$1$};
        \fill (p2) circle (2pt); \node[below] at (p2) {$2$};
        \fill (p3) circle (2pt); \node[above] at (p3) {$3$};
        \fill (p4) circle (2pt); \node[left]  at (p4) {$4$};
        \fill (p5) circle (2pt); \node[right] at (p5) {$5$};
        \fill (p6) circle (2pt); \node[above right,yshift=2pt] at (p6) {$6$};
        \end{tikzpicture}

        \smallskip
        The Fano matroid $F_7$.
    \end{minipage}
    \hfill
    \begin{minipage}{0.45\textwidth}
        \centering
        \begin{tikzpicture}[scale=0.9]
            \tikzstyle{edges}=[thick];
            \coordinate (p3) at (0,2.6);
            \coordinate (p4) at (-2.25,-1.3);
            \coordinate (p5) at (2.25,-1.3);
            \coordinate (p0) at (-1.125,0.65);
            \coordinate (p1) at (1.125,0.65);
            \coordinate (p2) at (0,-1.3);
            \coordinate (p6) at (0,0);

            \draw[thick] (p3)--(p4); 
            \draw[thick] (p3)--(p5); 
            \draw[thick] (p4)--(p5); 
            \draw[thick] (p3)--(p2); 
            \draw[thick] (p4)--(p1); 
            \draw[thick] (p5)--(p0); 

            \fill (p0) circle (2pt); \node[left]  at (p0) {$0$};
        \fill (p1) circle (2pt); \node[right] at (p1) {$1$};
        \fill (p2) circle (2pt); \node[below] at (p2) {$2$};
        \fill (p3) circle (2pt); \node[above] at (p3) {$3$};
        \fill (p4) circle (2pt); \node[left]  at (p4) {$4$};
        \fill (p5) circle (2pt); \node[right] at (p5) {$5$};
        \fill (p6) circle (2pt); \node[above right,yshift=2pt] at (p6) {$6$};
        \end{tikzpicture}

        \smallskip
        The non-Fano matroid $F_7^{-}$.
    \end{minipage}
    \end{center}
\end{example}
\smallskip

A \emph{triangulation} of a point configuration $A$ \cite[Def.~2.3.1]{DeLoeraRambauSantos} is a set $\Delta$ of simplices with vertices in~$A$ with
\begin{enumerate}
\item[(CP)] $\Delta$ is closed under taking faces.
\item[(UP)] $\bigcup_{\sigma \in \Delta} \sigma = \conv(A)$
\item[(IP)] $\relint(\sigma) \cap \relint(\tau) = \emptyset$ for all $\sigma, \tau \in \Delta$.
\end{enumerate}
Here, $\relint \sigma$ is the \emph{relative interior} of~$\sigma$, i.e., $\sigma$ minus the union of its proper faces.  Two simplices with vertices in~$A$ are said to \emph{intersect properly} if they, together with the set of all their faces, satisfy (IP).  A set of simplices with vertices in~$A$ whose union contains $\conv(A)$ is said to \emph{cover}~$A$.  With this, a triangulation of~$A$ is a simplicial complex with vertices in~$A$ so that all its maximal simplices intersect properly and cover $\conv(A)$.

In any triangulation $\Delta$ of~$A$, the following holds by (IP): For each point~$x \in \conv(A)$ there is a unique simplex $\sigma \in \Delta$ with $x \in \relint(\sigma)$.  This unique simplex is called \emph{the carrier of~$x$ in~$\Delta$} \cite[specialization of Def.~2.1.22]{DeLoeraRambauSantos}.

A full-dimensional lattice simplex is \emph{unimodular} if its normalized lattice volume is $1$ \cite{haase2021existence}.  In that case, by Cramer's rule, each lattice point is a unique integer affine combination of its vertices.  A triangulation is \emph{unimodular} if all its maximal simplices are unimodular.  A triangulation is \emph{regular} if it is the lower envelope complex that arises when a height is assigned to each vertex \cite[Def.~2.2.3]{DeLoeraRambauSantos}.  A triangulation is \emph{flag} if it is the clique complex of its graph or, equivalently, if every minimal non-face has two vertices \cite{athanasiadis2011some}. 

\begin{lemma}
    If a polytope $P$ has a regular unimodular flag triangulation, then any face of $F\subseteq P$ has a regular unimodular flag triangulation. 
\end{lemma}

\begin{proof}
    Fix a face $F\subseteq P$. Any triangulation $\Delta$ of $P$ induces a triangulation of $F$ by considering the set $\Delta|_F$ consisting of the simplices in $\Delta$ strictly contained in $F$. Regularity is inherited: if $\Delta$ comes from a height function, then restricting the same height function to $F$ induces the triangulation $\Delta|_F$. Unimodularity is also inherited, since $\Delta|_F \subseteq \Delta$. Finally, flagness is inherited, because $\Delta|_F$ is an induced subcomplex of $\Delta$, and induced subcomplexes of flag complexes are flag. 
\end{proof}

\begin{proposition}
    Let $M$ be a matroid whose base polytope admits a regular unimodular flag triangulation. Then, every minor $N$ of $M$ has a base polytope admitting a regular unimodular triangulation.
\end{proposition}

\begin{proof}
    Consider an element $e$ of the ground set of $M$. If $e$ is a loop or a coloop, then $M/e = M\setminus e$. If $e$ is a loop, the base polytope of $M$ is the base polytope of $M/e$ with an extra zero coordinate in the position labelled by $e$, thus the regular unimodular flag triangulation of $M$ gives directly a regular unimodular flag triangulation of $M/e=M\setminus e$. If $e$ is a coloop, the proof is completely analogous.

    Now, assume that $e$ is neither a loop nor a coloop, then the base polytope of $M\setminus e$ is identified with the intersection of $\mathcal{P}(M)$ with the hyperplane $\{x_e = 0\}$. Similarly, $\mathcal{P}(M/e)$ corresponds to the intersection $\mathcal{P}(M) \cap \{x_e = 1\}$. These two polytopes are faces of $\mathcal{P}(M)$ and thus, by the previous lemma, the regular unimodular flag triangulation of $\mathcal{P}(M)$ induces flag regular unimodular triangulations in them. Since all minors of $M$ are obtained by a sequence of deletions and contractions, this shows that in fact all minors of $M$ possess the same property.
\end{proof}

The above proposition implies that the class of matroids admitting a regular unimodular flag triangulation is closed under minors. Since the base polytope of any matroid $M$ is unimodularly equivalent to the base polytope of its dual $M^*$ via the map $\mathbb{R}^E\to \mathbb{R}^E$ given by $x \mapsto (1,\ldots,1) - x$, then the property of having regular unimodular flag triangulations is also closed under matroid duality. 

\begin{corollary}
    There exists a (possibly infinite) duality-closed list of matroids $\mathcal{U}$ with the following property: the base polytope of a matroid $M$ has no regular unimodular flag triangulation if and only if $M$ does not have minors belonging to $\mathcal{U}$.
\end{corollary}

\section{A method for enumerating all unimodular flag triangulations}
\label{sec:TOPCOM}

One obvious approach to decide the existence of a regular unimodular flag triangulation for a given point configuration is to enumerate \emph{all} triangulations first and search for a regular unimodular triangulation in the result. This has been tried unsuccessfully for example by \cite{Hayase2022}. The latest and currently fastest method for enumerating \emph{all} triangulations of a general point configuration $P$ up to symmetry is \emph{a symmetric lexicographic subset reverse search (symLSRS)} from \cite{Rambau:TOPCOM-preprint:2026}. 

Here is a short description of symLSRS.  The algorithm traverses an enumeration tree in a depth-first search manner.  The tree, rooted at the empty set, contains as nodes all \emph{partial triangulations} (i.e., subsets of maximal simplices that are intersect pairwise properly) that are lex-minimal (= lexicographically minimal) in their respective symmetry orbits.  For a node $T$, an outgoing edge corresponds to the addition of an \emph{admissible} simplex to~$T$.  A maximal simplex is admissible for~$T$ if it is pairwise intersecting properly with $T$ and lex-larger (= lexicographically larger) than each simplex in~$T$ so far.  Whenever a new node is not lex-minimal in its orbit, the node is discarded, and the algorithm backtracks. A partial triangulation is a complete triangulation if and only if all interior facets of maximal simplices in~$T$ are \emph{covered} (i.e. are contained in exactly two distinct maximal simplices, the \emph{pseudo-manifold property}).

Not every partial triangulation has a completion to a full triangulation by only admissible simplices (a \emph{right-completion}). Eventually, this can be detected when there are uncovered interior facets but no admissible simplex anymore.  Since there are many more not-right-completable partial triangulations than triangulations, a pruning criterion is crucial. Thus, the following necessary condition for right-completability is checked: If the lex-minimal uncovered facet is lex-smaller than the lex-minimal simplex admissible for~$T$, then $T$ cannot be right-completed to a triangulation.  With this blazingly fast so-called \emph{lex-pruning}, dead-ends can be detected much earlier.
The experiments in \cite{Rambau:TOPCOM-preprint:2026} show that with symLSRS much larger instances can be enumerated than before, where the flip-graph of triangulations was usually explored.  Another advantage is that restrictions like unimodularity can be implemented in symLSRS by only considering unimodular simplices as admissible simplices from the beginning.

Since, at present, the connectivity of the flip-graph of triangulations of matroid base polytopes is an open problem, flip-based explorations cannot decide the existence of a triangulation with certain properties of a matroid polytope.  Therefore, symLSRS in TOPCOM is currently the only implementation to find a conclusive answer by complete enumeration.

Preliminary computations for the Fano matroid polytope with $28$ points in rank~$7$ have shown that an enumeration of all unimodular triangulations and filtering by the flag and regularity properties ex-post seemed to take too long: After several days of HPC computations on 192 threads, it could be observed that the number of open nodes in each daily checkpoint was still increasing rapidly -- a sign for a very long way to go.  Thus, the computation was interrupted after $319{,}712{,}733$ symmetry classes of unimodular triangulations had been found.

Instead, a new method was developed to incorporate the flag property into the enumeration process.  While unimodularity was already considered and favorably tested in \cite{Rambau:TOPCOM-preprint:2026} (it makes enumeration faster by reducing the number of admissible simplices), it was not discussed yet how exactly the flag property could be handled in symLSRS.  The new approach is based on a criterion from \cite{Betre+etal:PureSimplicialComplexes-arxiv:2024}: A simplicial complex~$\Delta$ is flag if and only if each union of the pairwise intersections of three distinct maximal simplices is a face in~$\Delta$.  Call the cliques in this criterion the \emph{critical cliques of~$\Delta$}, and call the criterion the \emph{critical-cliques criterion}.

The critical-cliques criterion can be checked incrementally during symLSRS as follows:  For a partial triangulation~$T$, let $U(T)$ be a set of all \emph{uncovered critical cliques} in the edge graph of~$T$, i.e., critical cliques that are not (yet) faces in~$T$.  Now assume that a new admissible simplex~$S$ is added to~$T$ to obtain~$T'$.  Then, the following is true.
\begin{lemma}[Flag-Pruning]
    Let $T$ be a partial triangulation with uncovered critical cliques $U(T)$ and let $T' = T \cup \{S\}$ for an admissible simplex~$S$.  Let $E'$ be the set of admissible simplices for~$T'$. Then, the following algorithm detects that $T$ is not right-completable to a flag triangulation (``RETURN(false)") or that $T$ might be right-completable (``RETURN(true, $U'$)"), where $U' = U'(T')$.
    \begin{enumerate}
    \item Initialize $U' = \emptyset$.
    \item For each $C \in U(T)$:
    \begin{enumerate}
    \item If $C$ is a face of~$S$, then continue with the next $C$.
    \item If $C$ is lex-smaller than the lex-minimal simplex in~$E'$, then RETURN(false)
    \item Else add $C$ to~$U'$.
    \end{enumerate}
    \item For each $S' \neq S''$ in $T$:
    \begin{enumerate}
    \item Let $C$ be the critical clique for $S, S', S''$.
    \item If $C$ has more than rank many elements or $C$ is lex-smaller than the lex-minimal simplex in~$E'$, then RETURN(false).
    \item Else if $C$ is a face in $T'$, then continue with the next~$C'$.
    \item Else add $C$ to $U'$
    \end{enumerate}
    \item RETURN(true, $U'$)\qed
    \end{enumerate}
\end{lemma} 
The advantage of this incremental flag-pruning is that only critical cliques based on the new simplex~$S$ are iterated over and that the iteration can be stopped as soon as a critical clique is found that cannot become a face in any right-completion of~$T'$.

Integrating flag-pruning with lex-pruning allows for quite fast enumeration of flag triangulations of any point configuration, coming from a lattice polytope or not.

The tiny experiments in Table~\ref{tab:flag-pruning} on a MacBook M1Max single-threaded show that for small examples the numbers of flag triangulations (up to symmetry) can be computed while visiting substantially fewer nodes with flag-pruning.  (Regularity checks were postponed because their enormous effort is dominating anything else.)

\begin{table}[htbp]
\begin{center}
\sffamily
\begin{tabular}{rlrrr}
\toprule
configuration & restriction & \# triang's & \# nodes & CPU-time/s\\
\midrule
$4$-cube & none            & 247,451 & 3,446,659 & 4.5\\
$4$-cube & flag            &  36,394 &   803,274 & 9.1\\
$4$-cube & flag-unimodular &  24,939 &   576,373 & 6.8\\
\midrule
$\Delta_6 \times \Delta_2$ & none            & 533,242 & 6,325,472 & 470\\
$\Delta_6 \times \Delta_2$ & flag            &  30,951 &   609,070 &  63\\
$\Delta_6 \times \Delta_2$ & flag-unimodular &  30,951 &   609,070 &  64\\
\bottomrule
\end{tabular}
\end{center}
\caption{Computational results for flag pruning in symLSRS}
\label{tab:flag-pruning}
\end{table}

Compared to the enumeration of all triangulations, the CPU-times for the enumeration of flag triangulations is larger for the $4$-cube and smaller for $\Delta_6 \times \Delta_2$ (where all triangulations are unimodular; the extra second is for checking all simplex volumes in preprocessing).  An explanation may be that about 15\% of the triangulations of the $4$-cube are flag, while only about 6\% of the triangulations of $\Delta_6 \times \Delta_2$ are flag.  Moreover, the number of simplices in a triangulation of the $4$-cube is larger than the number of simplices in a triangulation of $\Delta_6 \times \Delta_2$, which leads to more calls of flag-pruning on the path to a triangulation.

Would flag-pruning accelerate the enumeration of flag triangulations of the Fano matroid polytope compared to the enumeration of all triangulations (recall that this was too slow)?  After several days of HPC-computation time on 192 threads and over 14,873 non-regular unimodular flag triangulations found, it seems that this method also takes too much time to get to a final verdict for the Fano matroid polytope.  Most probably this is due to the large size of a unimodular triangulation with 232 simplices (see below). We pushed to the limit the enumeration approach.  However, the collected evidence gave us confidence to attempt to prove that regular unimodular flag triangulations of the Fano matroid polytope do not exist.

Thus, in the following section, an entirely new method is described that does not seek to enumerate regular unimodular flag triangulations at all: It is aimed directly at \emph{deciding the existence} of a regular unimodular flag triangulation for a given \emph{lattice polytope}.

\section{Boolean constraints for regular unimodular flag triangulations}

In this section we now present a very general fast method that, given any lattice polytope $P$ decides whether it admits a regular unimodular flag triangulation. The key ideas used were a lemma that connects the $1$-skeleton of a lattice polytope to the lattice points in the dilations, a way to encode the skeleta of triangulations with Boolean formulas and SAT solvers, and symmetry breaking arguments. Later  we will apply it to matroid base polytopes.

\subsection{From Lattice Points inside Dilations to the Skeleton of Unimodular Triangulations}\label{idea-latticepts2triangs}

The following is the crucial object for our successful method.
\begin{definition}[Level-$t$ Decomposition of Lattice Points]
    Let $P=\conv(A)$ be a lattice polytope with vertex set
    \[
        A=\{v_1,\ldots,v_N\}\subset\ZZ^d. 
    \]
    For $t\ge 1$ and $y\in tP\cap\ZZ^d$, a \emph{level-$t$ decomposition of $y$} is an equality
    \[ 
        y=a_1v_{i_1}+\cdots+a_rv_{i_r},\qquad
        a_j\in\ZZ_{>0},\qquad \sum_{j=1}^r a_j=t,
    \]
    where the vertices $v_{i_1},\ldots,v_{i_r}$ with $i_1, \dots, i_r \in \{1, \dots, N\}$ are pairwise distinct.  The \emph{support} of this decomposition is
    \[
        F=\{v_{i_1},\ldots,v_{i_r}\}.
    \]
    Denote by $\cF_t(y)$ the set of all supports of level-$t$ decompositions of~$y$.
\end{definition}

To begin, here is the pseudocode of the algorithm to find a regular unimodular flag triangulation.

\begin{algorithm}[htbp]
\caption{SAT for regular unimodular flag triangulations}
\label{alg:carrier-regularity}
\begin{algorithmic}[1]
\Require $P$ a lattice polytope, $D \in \mathbb{N}$
\For{\(t=2,3,\ldots,D\)}
    \State Compute the sets \(\mathcal F_t(y)\) for all relevant lattice points
    \(y\in tP\cap\mathbb Z^d\).

    \For{each relevant lattice point \(y\in tP\cap\mathbb Z^d\)}
        \State Encode the \emph{carrier condition} 
        \State ``exactly one \(F\in\mathcal F_t(y)\) is a clique'' (see below)
        \State as a Boolean constraint in edge variables.
    \EndFor
\EndFor

\State Let \(\Phi\) be the resulting SAT formula.

\While{\(\Phi\) is satisfiable}
    \State Let \(G\) be a candidate graph satisfying these clauses.
    \State Test the necessary regularity inequalities for \(G\).

    \If{the inequalities are feasible}
        \State \textbf{return} \(G\).
    \Else
        \State Extract a subset of edges that is incoherent with regularity.
        \State Add the corresponding blocking clause (up to symmetry) to \(\Phi\).
    \EndIf
\EndWhile

\State \textbf{return} UNSAT.
\end{algorithmic}
\end{algorithm}

\FloatBarrier

If Algorithm \ref{alg:carrier-regularity} returns UNSAT, then no regular unimodular flag triangulation exists.  If it returns $G$, its clique complex must be checked independently as a geometric triangulation.

If $D = 2$, the SAT formula is independent of flagness; if it is unsatisfiable, then no unimodular triangulation exists. We now explain each part of the algorithm in detail.  

The following lemma is reminiscent of the fact that from each circuit of a point configuration only its positive part or its negative part can form a face in the same triangulation, not both.  However, beyond \emph{preventing} faces if other faces are present in a triangulation, the following lemma also \emph{enforces} faces if other faces are not present in a triangulation.
We wish to stress the strong interdependence between a unimodular triangulation of a lattice polytope on the one hand and its lattice points and its Ehrhart function on the other hand.  This has already been reflected, e.g., in \cite{SamPayne2008} and in \cite[Theorem 5.5]{MaclaganThomas2002}.

\begin{lemma}[Carrier Lemma for Unimodular Triangulations of Lattice Polytopes]\label{lem:carrier}
    Let $P$ be a lattice polytope and let $\Delta$ be a unimodular triangulation of $P$.  Then, for all $t \ge 1$ and all $y \in tP \cap \ZZ^d$, $\Delta$ contains the convex hull of exactly one support face $F \in \mathcal{F}_t(y)$ of a level-$t$ decomposition of~$y$.
\end{lemma}

\begin{proof}
    The point $x := \frac{1}{t}\cdot y$ lies in~$P$.  Call $\tau$ the carrier of $x$ in~$\Delta$, i.e., $x \in \relint(\tau)$, and there is no other $\tau' \in \Delta$ with this property.  Let $F = \{v_{i_1},\ldots,v_{i_r}\}$ be the unique set of vertices of~$\tau = \conv(F)$ (the vertices of a polytope do not contain duplicate points).  Then, $x \in \relint(\tau)$ implies that $x$ has a representation as $x = \frac{1}{t} \cdot y = \lambda_1 v_{i_1} + \dots + \lambda_r v_{i_r}$ for unique positive barycentric coordinates $\lambda_1, \dots, \lambda_r$ with $\lambda_1 + \dots + \lambda_r = 1$.  Thus, $y = t\lambda_1 v_{i_1} + \dots + t\lambda_r v_{i_r}$.  This can be extended to $x = t\lambda_1 v_{i_1} + \dots + t\lambda_r v_{i_r} + t \cdot 0 \cdot v_{i_{r+1}} + \dots + t \cdot 0 \cdot v_{i_{d+1}}$ for some maximal simplex $\sigma = \conv(v_{i_1}, \dots, v_{i_r}, v_{i_{r+1}}, \dots, v_{i_{d+1}})$ in~$\Delta$ containing~$\tau$.  The vertices of~$\sigma$ are affinely independent, hence, this representation is unique.  Since $y$ is a lattice point and the vertices of~$\sigma$ (a full-dimensional unimodular simplex) generate the lattice via integer affine combinations, the (unique) representation of~$y$ as an affine combination of vertices of~$\sigma$ as above must have integer coefficients. Now, define $a_1 := t \lambda_1 > 0, \dots , a_r := t \lambda_r > 0$.  Then, $a_j \in \ZZ_{>0}$ for all $j = 1, \dots, r$, $y = a_1v_{i_1} + \dots + a_r v_{i_r}$, and $a_1 + \dots + a_r = t\lambda_1 + \dots + t\lambda_r = t(\lambda_1 + \dots + \lambda_r) = t$.  Any sum of positive integers adding up to~$t$ has at most $t$ summands. Thus, $F$ is the support of a level-$t$ decomposition of~$y$, $\abs{F} \le t$, and $\conv(F) = \tau$ is a face of~$\Delta$.
 
    If $G \in \mathcal{F}_t(y)$ with $F \neq G$ and $y = b_1 v_{j_1} + \dots + b_r v_{j_r}$ such that $b_1 + \dots + b_r = t$, then $x = \frac{1}{t} \cdot y = \frac{b_1}{t} v_{j_1} + \dots + \frac{b_r}{t} v_{j_r}$, where all coefficients are positive and sum to $\frac{t}{t} = 1$. In other words: $x \in \relint(\tau')$ with $\tau' = \conv(G) \neq \tau$.  Therefore, $\tau' = \conv(G)$ cannot be a face of~$\Delta$ by the uniqueness of the carrier of~$x$ in~$\Delta$.
\end{proof}

The lemma says that $tP$ and the lattice points in it detect possible and necessary carrier faces of a unimodular triangulation with at most $t$ vertices. Thus, level $2$ detects vertices and edges; level $3$ detects vertices, edges, and triangles; level $4$ detects vertices, edges, triangles, and tetrahedra, etc.

The flag condition is used to express these possible carrier faces using only the graph:
\[
 F\in T
 \quad\Longleftrightarrow\quad
 F\text{ is a clique in the one-skeleton of }T.
\]
Therefore, for a flag unimodular triangulation, the carrier lemma becomes the Boolean rule
\[
\boxed{
\text{for every }y\in tP\cap\ZZ^d,
\text{ exactly one }F\in\cF_t(y)\text{ is a clique.}
}
\tag{1}
\label{eq:carrier-rule}
\]

Again, note that level $2$ does not depend on flagness.  If $y=u+v$ with $u\ne v$ vertices, then $\frac{1}{2}\cdot y$ cannot have a vertex carrier, since a vertex of $P$ is extreme. Hence, the carrier must be an edge.  Thus, in every nontrivial midpoint fiber
\[
 \{\{u,v\}:u+v=y\}
\]
exactly one pair is an edge of the triangulation. 

\subsection{SAT encoding of flag unimodular triangulations through carriers}\label{sec:sat}

We encode \eqref{eq:carrier-rule} as a Boolean formula.  The variables describe the graph that would be the 1-skeleton of a unimodular flag triangulation.  The formula does not initially encode maximal simplices; these are recovered, if possible,
as cliques of the graph.

For every unordered pair \(\{i,j\}\) of vertices we introduce a Boolean
variable \(e_{ij}=e_{ji}\), interpreted as the assertion that
\(\{v_i,v_j\}\) is an edge.  A set \(F\) of vertices is a face of a flag
complex precisely when all pairs in \(F\) are edges.  Thus, for example,
\[
  \{i,j,k\}\text{ is a face}
  \quad\Longleftrightarrow\quad
  e_{ij}\wedge e_{ik}\wedge e_{jk}.
\]

Suppose, for instance, that a point \(y\in 3P\) has exactly two possible carriers: an edge \(uv\) and a triangle \(ijk\).  The carrier condition says that exactly one of these two possibilities occurs:
\[
\bigl(e_{uv}\wedge \neg(e_{ij}\wedge e_{ik}\wedge e_{jk})\bigr)
\vee
\bigl((e_{ij}\wedge e_{ik}\wedge e_{jk})\wedge \neg e_{uv}\bigr).
\]
In the implementation this is converted to CNF.

\begin{definition}
Fix a level bound $D\ge 2$.  A set of edges $E$ on the vertex set $A$ is called \emph{$D$-consistent} if rule \eqref{eq:carrier-rule} holds for every $2\le t\le D$ and every lattice point $y\in tP$ represented by vertices.  
\end{definition}

An equivalent SAT formulation of $D$-consistency is the conjunction of exactly-one clauses encoding rule \eqref{eq:carrier-rule} for the relevant lattice points. For any lattice point $y \in tP$, rule \eqref{eq:carrier-rule} is equivalent to the following Boolean constraint. For Boolean formulas \(\phi_1,\ldots,\phi_m\), let the exactly-one clause be
\[
\operatorname{EO}(\phi_1,\ldots,\phi_m)
:=
(\phi_1\vee\cdots\vee\phi_m)
\wedge
\bigwedge_{a<b}(\neg\phi_a\vee\neg\phi_b).
\label{eq:exactly-one-clause}
\]

For \(F\subseteq A\), set
\[
  \kappa_F:=\bigwedge_{\{u,v\}\subseteq F}e_{uv}.
\]
Thus \(\kappa_F\) asserts that \(F\) is a clique.  The carrier constraint \eqref{alg:carrier-regularity} for \(y\in tP\cap\mathbb Z^d\) is the Boolean formula
\[
  \operatorname{EO}\bigl(\kappa_F:F\in\mathcal F_t(y)\bigr).
\]

Every unimodular flag triangulation maps to a $D$-consistent edge set for every $D$.  On the other hand, a $D$-consistent edge set is only a necessary condition for the existence of a unimodular flag triangulation. 

\subsection{Ruling out regularity of unimodular flag triangulations}\label{flaguni-implies-nonregular}

We now introduce necessary conditions for regularity of a triangulation. Let
\(A=\{v_1,\ldots,v_N\}\) be the vertex set of a lattice polytope \(P\).
Recall a \emph{regular triangulation} is induced by a height vector
\[
  h=(h_1,\ldots,h_N)\in\mathbb R^N.
\]
Geometrically, one lifts \(v_i\) to \((v_i,h_i)\), takes the lower hull, and projects its lower faces back to \(P\). It is also well-known that regularity can be verified via a linear programming problem (see \cite{DeLoeraRambauSantos}). 

In our case the only regularity inequalities used in the certificate for \(P_{F_7}\) come from level \(t = 2\) in Algorithm \ref{alg:carrier-regularity}. Suppose that a level-$2$ fiber contains two decompositions of the same lattice point:
\[
  v_i+v_j=v_k+v_\ell .
\]
If the triangulation selects the edge \(\{i,j\}\) as the carrier of this midpoint,
then regularity forces this carrier to be lower than every other pair in the same fiber:
\[
  h_i+h_j<h_k+h_\ell .
\tag{2}
\]
Since only finitely many such strict inequalities occur, strict feasibility is
equivalent, after a rational perturbation and scaling, to feasibility with margin
\(1\):
\[
  h_i+h_j-h_k-h_\ell\le -1.
\label{eq:height-margin}
\tag{3}
\]

In Algorithm \ref{alg:carrier-regularity}, the search alternates between SAT and exact linear infeasibility. The SAT solver returns a graph satisfying the carrier clauses. From this graph we read the selected edge in every level-$2$ exactly-one clause and test \eqref{eq:height-margin} with a linear program (LP).  If the system is feasible, the graph survives this necessary regularity test.  If the LP is infeasible, we extract an additional Boolean constraint from a small set of
selected edges
\[
 C=\{(\phi_1,e_1),\ldots,(\phi_m,e_m)\},
\]
where each $\phi_i$ is a levele-2 carrier clause, and $e_i$ is the selected edge in $\phi_i$, such that $C$ already makes the height inequalities impossible.  We call $C$ a \emph{forbidden height pattern}. We add the blocking clauses corresponding to $C$ to the carrier clauses, run the SAT solver with the additional blocking clauses and repeat the loop until we either find a regular unimodular flag triangulation or we reach UNSAT.

\begin{definition}[Quadratic Coherence]
A \(D\)-consistent graph is \emph{quadratically coherent} if there exists a height vector \(h\in\mathbb R^N\) satisfying \((3)\) for every selected edge and every competing option in each carrier clause.
\end{definition}

\begin{remark}
Quadratic coherence is only a necessary condition for regularity. This is enough for our purpose: if no \(D\)-consistent graph is quadratically coherent, then no regular unimodular flag triangulation can exist.
\end{remark}

\begin{lemma}\label{lem:regular-implies-coherent}
Every regular unimodular flag triangulation gives a \(D\)-consistent and
quadratically coherent graph, for every \(D\ge 2\).
\end{lemma}

The certificate for such a pattern is exact.  Each comparison inequality has the form
\[
 r_a\cdot h\le -1,
\]
where $r_a\in\ZZ^{N}$ has two $+1$ entries and two $-1$ entries.  The stored Farkas certificate is a list of nonnegative integers $\lambda_a$, not all zero, such that
\[
 \sum_a\lambda_a r_a=0.
\tag{4}
\]
If a height vector satisfied all inequalities, then multiplying by $\lambda_a$ and summing would give
\[
0=\left(\sum_a\lambda_a r_a\right)\cdot h
\le -\sum_a\lambda_a<0,
\]
which is impossible. Therefore, no regular triangulation can contain all choices in $C$.

The corresponding SAT blocking clause is
\[
 \neg z_{\phi_1,e_1}\vee\cdots\vee \neg z_{\phi_m,e_m}.
\tag{5}
\]
This says that at least one of the choices in the forbidden height pattern must change. This clause blocks many assignments at once, not only the single assignment that led to the pattern.

Repeating this procedure terminates because there are only finitely many assignments of the selector variables.  In the Fano computation, termination occurs by unsatisfiability in each of the six symmetry-broken cases.  Since each added blocking clause is justified by an exact Farkas certificate, the final UNSAT result proves that no \(D\)-consistent graph can be quadratically coherent.  By Lemma \ref{lem:regular-implies-coherent}, this rules out regular unimodular flag triangulations.

\begin{remark}[Symmetry Breaking]
Suppose that a group \(\Gamma\) acts on \(P=\conv(A)\) by lattice
automorphisms preserving \(A\).  Then \(\Gamma\) acts on the edge variables,
carrier clauses, and height inequalities by relabeling:
\[
  e_{ij}\mapsto e_{\gamma(i)\gamma(j)},
  \qquad
  F\in\cF_t(y)\mapsto \gamma F\in\cF_t(\gamma y).
\]
Thus the carrier clauses and the quadratic coherence inequalities are
\(\Gamma\)-invariant.  In particular, if a \(D\)-consistent graph is
quadratically coherent, then every translate of it is also \(D\)-consistent
and quadratically coherent.

We may therefore impose symmetry-breaking assumptions whenever they choose at
least one representative from each \(\Gamma\)-orbit of possible solutions.
Similarly, if a forbidden height pattern \(C\) is certified impossible, then
every translate \(\gamma C\) is impossible as well.  Hence the full
\(\Gamma\)-orbit of the corresponding blocking clause may be added to the SAT
formula.
\end{remark}

\subsection{Certificate structure and independent verification}

The SAT solver is used only to search.  The proof of nonexistence is the
combination of the following independently checkable data:
\begin{enumerate}
\item The CNF encoding of the level-\(2\) and level-\(3\) carrier constraints.
\item The symmetry-breaking assumptions.
\item For every added blocking clause, an exact Farkas certificate proving that the blocked height pattern is incompatible with quadratic coherence.
\item The final UNSAT certificate for each of the symmetry-broken CNF instances.
\item A deterministic checker verifying that all clauses and all orbit cuts are valid consequences of the data.
\end{enumerate}

\section{Using the method on the Fano matroid polytope}\label{largermatroids}

In the rest of this section we consider only matroid polytopes, in fact we concentrate in the base polytope of the Fano matroid. It is known that it has the following Ehrhart $h^*$-polynomial (see \cite[Table~2]{DeLoeraHawsKoeppe}),
\[
 h^*(P_{F_7})=(1,21,91,98,21,0,0).
\]
The total sum of these entries is $232$. Since the sum of the $h^*$-vector of a polytope agrees with the normalized volume of the polytope, we have that every unimodular triangulation of $P_{F_7}$ has exactly $232$ maximal $6$-simplices. More generally, if a polytope $P$ admits a unimodular triangulation, the $h$-polynomial of the triangulation agrees with the $h^*$-vector of the polytope (see \cite[Theorem~10.3]{BeckRobins}). In particular, all the unimodular triangulations of the base polytope of $F_7$ have the same $f$-vector, which is obtained from the $h$-vector by a linear transformation:
    \[ (28,238,903,1778,1897,1043,232).\]

For the Fano matroid base polytope, the level-$2$ carrier clauses of distinct pairs have the following sizes:
\[
126\text{ singleton carrier clauses},\quad 105\text{ two-way carrier clauses},\quad 7\text{ six-way carrier clauses}.
\]
The singleton carrier clauses give forced edges.  If a carrier clause $\phi$ has options $e=1,\ldots,s_{\phi}$, then the Boolean variable
\[
 z_{\phi,e}
\]
means that edge $e$ is the chosen edge in carrier clause $\phi$.  The exact-one clauses are
\[
 z_{\phi,1}\vee\cdots\vee z_{g,m_g},
\qquad
 \neg z_{\phi,e}\vee\neg z_{\phi,e'}\quad(e\ne e').
\]
There are $112$ non-forced carrier clauses.  A solution of the quadratic clauses therefore gives
\[
126+112=238
\]
edges. All remaining pairs are nonedges.

The next proposition shows that the obstruction is genuinely one of regularity and also serves as an empirical verification that the algorithm works before the regularity step. The Fano matroid base polytope does admit unimodular flag triangulations; none of them can be regular by the main theorem. The certificate was found using Algorithm~\ref{alg:carrier-regularity} before the regularity steps. It can be found in the appendix.  Recall that in Section~\ref{sec:TOPCOM} it was already reported that over 14,873 non-regular unimodular flag triangulations of the Fano matroid base polytope have been found independently by TOPCOM.

\begin{proposition}\label{prop:positive}
The Fano matroid base polytope admits unimodular flag triangulations.
\end{proposition}

\begin{proof}
The certificate is a graph $G$ on the $28$ Fano bases.  Its clique complex
\[
 \Delta(G)=\{S:S\text{ is a clique of }G\}
\]
has face vector
\[
 f(\Delta(G))=(28,238,903,1778,1897,1043,232).
\]
The independent checker verifies that there are no $8$-cliques; that the maximal cliques are exactly $232$ seven-vertex sets; that each maximal clique spans a unimodular $6$-simplex; and that every pair of maximal simplices intersects in the convex hull of its common vertices.  Since the total normalized volume is $232=\Vol_{\ZZ}(P_{F_7})$, these simplices cover $P_{F_7}$.  Thus $\Delta(G)$ is a triangulation.  It is flag by definition, because it is a clique complex, and it is unimodular because all maximal simplices are unimodular.
\end{proof}

The same certificate found in the Appendix can also be written in TOPCOM \cite{Rambau2002TOPCOM} seed format using the usual $0/1$ coordinates in the hyperplane $x_0+\cdots+x_6=3$ as coordinates for the vertices of the Fano matroid base polytope.  The independent Python verifier checks the clique-complex definition of flagness and the geometric triangulation condition.  TOPCOM in versions 1.2.0 and above \cite{Rambau:TOPCOM-preprint:2026} can provide an independent confirmation of the triangulation property, unimodularity, and flagness of the seed triangulation from scratch.

\subsection{Symmetry and the six cases} We use symmetry-breaking and Lemma \ref{lem:regular-implies-coherent} with $D = 3$ to show that there is no regular unimodular flag triangulation of the Fano matroid. 

The automorphism group of the Fano matroid has order $168$.  It acts on bases, selected edges, level-$2$ carrier clauses, and height inequalities.  Among the level-$2$ carrier clauses there are seven six-way clauses, with $42$ total edges. The automorphism group acts transitively on these $42$ edges. We have found an instance for which the answer is $42$.

One of the edges is
\[
 013\mid 246.
\]

This notation means the graph edge between the Fano bases $013$ and $246$.  These two bases are disjoint and together use all ground-set elements except $5$. Since every solution chooses one option in each six-way set, and since all $42$ such options lie in one automorphism orbit, any hypothetical solution can be relabeled so that it contains $013\mid246$.  We impose this as a symmetry break by setting this $e_{013, 246}$ as \texttt{true}.

After this, we split on the six cases of another six-way clause.  In the implementation, it is the carrier clause whose options are
\[
013\mid256,
016\mid235,
023\mid156,
025\mid136,
035\mid126,
036\mid125.
\]
Every remaining solution chooses exactly one of these six edges. Hence checking the six cases exhausts the symmetry-broken search space.  The choice of this set is not mathematically meaningful, it is a choice that proved successful.

Whenever a forbidden height pattern is found, the proof adds its full orbit under $\Aut(F_7)$.  Running algorithm \ref{alg:carrier-regularity} with $D = 3$ on the six cases gives the following results:

\[
\begin{array}{c|c|c|c|c}
\text{case} & \text{fixed option} & \text{patterns} & \text{orbit cuts} & \text{status}\\
\hline
0 & 013\mid256 & 17 & 2856 & \mathrm{UNSAT}\\
1 & 016\mid235 & 16 & 2688 & \mathrm{UNSAT}\\
2 & 023\mid156 & 20 & 3360 & \mathrm{UNSAT}\\
3 & 025\mid136 & 10 & 1680 & \mathrm{UNSAT}\\
4 & 035\mid126 & 26 & 4368 & \mathrm{UNSAT}\\
5 & 036\mid125 & 13 & 2184 & \mathrm{UNSAT}
\end{array}
\]
In total, the certificate stores $102$ representative forbidden height patterns.  Their automorphism orbits produce $17136$ blocking clauses. This implies the main theorem of this paper. The time to produce the certificate was $14$ seconds on an AMD Ryzen 9 9950X3D (16 cores, 32 threads), 64 GB RAM, Ubuntu 24.04, Python 3.13. The SAT computations were performed with CaDiCaL version 1.7.3 \cite{CaDiCaL173}, using the solver described in
\cite{BiereFazekasFleuryHeisinger-SAT-Competition-2020-solvers}.

\begin{theorem}\label{thm:no-regular}
The Fano matroid base polytope has no regular unimodular flag triangulation.
\end{theorem}

\begin{remark}
    As is clear from the preceding discussion, our proof relies on several subtle computations carried out by a computer. However, we wish to emphasize that none of the mathematical ideas discussed in the present paper, which ultimately lead to the main result appearing above, relies on the use of Artificial Intelligence.
\end{remark}

\section{Conclusions and further comments}

\subsection{Code to generate concrete examples}
Our code implements Algorithm \ref{alg:carrier-regularity} and proves and verifies Theorem \ref{thm:no-regular}. It will be provided upon request. The same code can help investigate whether other matroid base polytopes have regular unimodular flag triangulations. The input data takes the following form. 


\begin{verbatim}
{
  "name": "Fano",
  "ground_set": [0, 1, 2, 3, 4, 5, 6],
  "rank": 3,
  "non_bases": [[0,1,2],[0,3,4],[0,5,6],
                [1,3,5],[1,4,6],[2,3,6],[2,4,5]]
}
\end{verbatim}

 We ran our code on several famous small matroids. 
Table \ref{tab:small-matroids} records the outcome of the included certificates for several well-known small matroids. The notation follows the standard matroid notation in Oxley \cite{Oxley}.

\begin{table}[htbp]
\centering
\small
\setlength{\tabcolsep}{4pt}
\renewcommand{\arraystretch}{1.12}
\begin{tabular}{@{}l>{\raggedright\arraybackslash}p{0.28\textwidth}rrrrrc@{}}
\toprule
Notation & Name & \(n\) & \(\operatorname{rank}\) & \(|\cB|\) &
\(\dim P\) & \(\Vol_{\ZZ}(P)\) &
\begin{tabular}[c]{@{}c@{}}quadratic\\triangulation?\end{tabular} \\
\midrule
\(U_{2,4}\) & four-point line & 4 & 2 & 6 & 3 & 4 & yes \\
\(U_{2,5}\) & five-point line & 5 & 2 & 10 & 4 & 11 & yes \\
\(U_{3,6}\) & uniform rank-three matroid & 6 & 3 & 20 & 5 & 66 & yes \\
\(M(K_4)\) & complete-graphic matroid & 6 & 3 & 16 & 5 & 42 & yes \\
\(W^3\) & rank-three whirl & 6 & 3 & 17 & 5 & 48 & yes \\
\(F_7\) & Fano matroid & 7 & 3 & 28 & 6 & 232 & {\bf no} \\
\(F_7^{-}\) & non-Fano matroid & 7 & 3 & 29 & 6 & 242 & yes \\
\(V_8\) & V\'amos matroid & 8 & 4 & 65 & 7 & 2316 & yes \\
\bottomrule
\end{tabular}
\caption{Matroid base polytopes tested.  Here \(n\) is the size of the ground set, \(\cB\) is the set of bases, and ``quadratic triangulation'' means regular unimodular flag triangulation.}
\label{tab:small-matroids}
\end{table}

\subsection{Koszulity and flagness}

Aside from asking about the existence of quadratic Gr\"obner bases, in \cite[p.~261]{HerzogHibi2002} Herzog and Hibi also inquire about a weaker property: that of being Koszul. Let $k$ be a field. A graded $k$-algebra $A$ is said to be a \emph{Koszul algebra} if the minimal free $A$-resolution of $k$ is linear. For a thorough survey on Koszulity, we refer to Fr\"oberg \cite{froberg}. 

\begin{question}[\cite{HerzogHibi2002}]\label{question:koszul}
    Is the base ring associated to a polymatroid a Koszul algebra?
\end{question}

Despite being a purely algebraic property, the Koszulity of a graded algebra has fascinating applications in combinatorics. For example, a classical result by Fr\"oberg shows that the Stanley--Reisner ring of a simplicial complex $\Delta$ is Koszul if and only if $\Delta$ is a flag complex.
A well-known result by Anick \cite{anick} (see \cite[Section~4]{froberg}) shows that if an ideal $I\subseteq k[x_1,\ldots,x_n]$ in a polynomial ring has a quadratic Gr\"obner basis, then $k[x_1,\ldots,x_n]/I$ is Koszul (see \cite[Corollary~2.1.3]{BrunsGubeladzeTrung} for a polyhedral discussion on this fact). Furthermore, the Koszul property  on $k[x_1,\ldots,x_n]/I$ implies that the ideal $I$ is quadratically generated. Although we disproved the existence of quadratic Gr\"obner bases for (poly)matroid toric ideals, we suspect that the answer to the Koszulity question by Herzog and Hibi might actually be affirmative.

As noted above, Koszulity is tightly related to the flagness property. It is also natural to inquire whether polymatroids always admit the existence of unimodular flag triangulations. We have not been able to find any matroid without a flag unimodular triangulation. Based on several experiments on (poly)matroids on a ground set of size $8$, we suspect that the following conjecture might be true.

\begin{conjecture}\label{conj:flag}
    Every polymatroid base polytope has a unimodular flag triangulation.
\end{conjecture}

As proved in this article, one cannot, in general, expect those triangulations to be regular. At this time, the relationship between the Koszul property and the existence of unimodular flag triangulations is very enigmatic. We emphasize that the latter property only means that the Stanley--Reisner ring of the triangulation is Koszul, but this property does not obviously transfer to the base ring of the matroid. We suggest \cite{BrunsGubeladzeTrung,Payne09} for further reading about polytopes and the Koszul property.

\subsection{Relationship with Ehrhart inequalities}

As we shall now explain, a proof of Conjecture~\ref{conj:flag} or an affirmative answer to Question~\ref{question:koszul} would have interesting consequences at the level of the Ehrhart polynomials. The base ring of a polymatroid is naturally isomorphic to the Ehrhart ring of the base polytope of the polymatroid (see \cite{MillerSturmfels,MichalekSturmfels} for more details). 

In general, Koszulity imposes heavy restrictions on the shape of the Hilbert series of an algebra.  As an example, the Hilbert series $H_A(t)\in \mathbb{Z}_{\geq 0}[[t]]$ of a Koszul algebra $A$ is known to satisfy several ``positivity properties''; for example, the infinite series $1/H_{A}(-t)$ is constrained to have nonnegative coefficients by \cite[Theorem~1]{froberg} --- further inequalities were established by Polishchuk and Positselski \cite[Chapter~7, Theorem 2.1]{PolishchukPositselski}.

As mentioned earlier in this paper, the Ehrhart $h^*$-polynomial of $P$ is equal to the $h$-vector of a triangulation~$\Delta$ of~$P$. Moreover, the existence of a unimodular flag (not necessarily regular) triangulation of a polytope also imposes inequalities at the level of $h^*$-polynomials: as explained above, the $h$-polynomial of any unimodular triangulation matches the $h^*$-polynomial of the polytope. In particular, the flagness of the triangulation imposes that the $h^*$-polynomial must be the $h$-vector of a flag simplicial complex. Several inequalities have been proved for such simplicial complexes; e.g., the Frankl--F\"uredi--Kalai inequalities \cite{frankl-furedi-kalai} (all of which apply to flag complexes thanks to Fr\"ohmader's results \cite{frohmader}), and an inequality proved by Stanley in the context of chromatic function on graphs (see \cite[p.~100]{stanley-green}).

In particular, proving Conjecture~\ref{conj:flag} or answering affirmatively Question~\ref{question:koszul} would impose several positivity properties for the Ehrhart series and the Ehrhart $h^*$-vector of matroid base polytopes. This is very much related to the unimodality conjecture for $h^*$-polynomials of matroids posed by De~Loera, Haws, and K\"oppe \cite[Conjecture~2b]{DeLoeraHawsKoeppe} and to the stronger version, asserting real-rootedness, posed by Ferroni in \cite[Conjecture~1.3]{Ferroni}. That real-rootedness conjecture, if proved true, would be compatible with the known numerical obstructions given by Koszulity (see \cite[Theorem~4.13]{ferroni-higashitani}), and it is also conjecturally compatible with the numerical obstructions imposed by the existence of unimodular flag triangulations (see \cite[Question~5.1]{bell-skandera}).

\bibliographystyle{amsalpha}
\bibliography{bibliography}

\newpage
\appendix 
\section{A unimodular flag triangulation of the Fano matroid}

As a sanity check and as an illustration of the edge selection, we record a unimodular flag triangulation $\Delta$ found by the SAT search: its $28$ vertices are the bases of the Fano matroid (the $3$-subsets of $\{0,\dots,6\}$ other than the seven non-bases $012$, $034$, $056$, $135$, $146$, $236$, $245$), $\Delta$ is the clique complex of the graph $G$ below, and an edge between bases $B,B'$ is written $B \mid B'$. We note we have found many unimodular flag triangulations; however, none of them are also regular.

\subsection*{The triangulation \texorpdfstring{$\Delta$}{Delta} (232 facets)}
Each facet is a maximal simplex of $\Delta$, written as the $7$ bases it spans.
\begin{center}\footnotesize
$\{$ $\{013,\,014,\,015,\,016,\,024,\,045,\,246\}$, $\{013,\,014,\,015,\,016,\,024,\,126,\,246\}$,\\
$\{013,\,014,\,015,\,016,\,045,\,246,\,456\}$, $\{013,\,014,\,015,\,016,\,126,\,156,\,246\}$,\\
$\{013,\,014,\,015,\,016,\,156,\,246,\,456\}$, $\{013,\,014,\,015,\,024,\,045,\,246,\,345\}$,\\
$\{013,\,014,\,015,\,024,\,123,\,126,\,246\}$, $\{013,\,014,\,015,\,024,\,123,\,134,\,246\}$,\\
$\{013,\,014,\,015,\,024,\,134,\,246,\,345\}$, $\{013,\,014,\,015,\,045,\,246,\,345,\,456\}$,\\
$\{013,\,014,\,015,\,123,\,126,\,156,\,246\}$, $\{013,\,014,\,015,\,123,\,134,\,156,\,246\}$,\\
$\{013,\,014,\,015,\,134,\,156,\,246,\,456\}$, $\{013,\,014,\,015,\,134,\,246,\,345,\,456\}$,\\
$\{013,\,014,\,016,\,024,\,045,\,046,\,246\}$, $\{013,\,014,\,016,\,045,\,046,\,246,\,456\}$,\\
$\{013,\,014,\,016,\,046,\,136,\,246,\,456\}$, $\{013,\,014,\,016,\,126,\,136,\,156,\,246\}$,\\
$\{013,\,014,\,016,\,136,\,156,\,246,\,456\}$, $\{013,\,014,\,024,\,045,\,046,\,246,\,345\}$,\\
$\{013,\,014,\,024,\,046,\,134,\,246,\,345\}$, $\{013,\,014,\,045,\,046,\,246,\,345,\,456\}$,\\
$\{013,\,014,\,046,\,134,\,136,\,246,\,456\}$, $\{013,\,014,\,046,\,134,\,246,\,345,\,456\}$,\\
$\{013,\,014,\,123,\,126,\,136,\,156,\,246\}$, $\{013,\,014,\,123,\,134,\,136,\,156,\,246\}$,\\
$\{013,\,014,\,134,\,136,\,156,\,246,\,456\}$, $\{013,\,015,\,016,\,024,\,025,\,045,\,246\}$,\\
$\{013,\,015,\,016,\,024,\,025,\,126,\,246\}$, $\{013,\,015,\,016,\,025,\,035,\,045,\,256\}$,\\
$\{013,\,015,\,016,\,025,\,045,\,246,\,256\}$, $\{013,\,015,\,016,\,025,\,126,\,246,\,256\}$,\\
$\{013,\,015,\,016,\,035,\,045,\,256,\,456\}$, $\{013,\,015,\,016,\,035,\,256,\,356,\,456\}$,\\
$\{013,\,015,\,016,\,045,\,246,\,256,\,456\}$, $\{013,\,015,\,016,\,126,\,156,\,246,\,256\}$,\\
$\{013,\,015,\,016,\,156,\,246,\,256,\,456\}$, $\{013,\,015,\,016,\,156,\,256,\,356,\,456\}$,\\
$\{013,\,015,\,024,\,025,\,045,\,246,\,345\}$, $\{013,\,015,\,024,\,025,\,123,\,126,\,246\}$,\\
$\{013,\,015,\,024,\,025,\,123,\,134,\,246\}$, $\{013,\,015,\,024,\,025,\,134,\,246,\,345\}$,\\
$\{013,\,015,\,025,\,035,\,045,\,256,\,345\}$, $\{013,\,015,\,025,\,035,\,235,\,256,\,345\}$,\\
$\{013,\,015,\,025,\,045,\,246,\,256,\,345\}$, $\{013,\,015,\,025,\,123,\,126,\,246,\,256\}$,\\
$\{013,\,015,\,025,\,123,\,134,\,235,\,256\}$, $\{013,\,015,\,025,\,123,\,134,\,246,\,256\}$,\\
$\{013,\,015,\,025,\,134,\,235,\,256,\,345\}$, $\{013,\,015,\,025,\,134,\,246,\,256,\,345\}$,\\
$\{013,\,015,\,035,\,045,\,256,\,345,\,456\}$, $\{013,\,015,\,035,\,235,\,256,\,345,\,356\}$,\\
$\{013,\,015,\,035,\,256,\,345,\,356,\,456\}$, $\{013,\,015,\,045,\,246,\,256,\,345,\,456\}$,\\
$\{013,\,015,\,123,\,126,\,156,\,246,\,256\}$, $\{013,\,015,\,123,\,134,\,156,\,246,\,256\}$,\\
$\{013,\,015,\,123,\,134,\,156,\,256,\,356\}$, $\{013,\,015,\,123,\,134,\,235,\,256,\,356\}$,\\
$\{013,\,015,\,134,\,156,\,246,\,256,\,456\}$, $\{013,\,015,\,134,\,156,\,256,\,356,\,456\}$,\\
$\{013,\,015,\,134,\,235,\,256,\,345,\,356\}$, $\{013,\,015,\,134,\,246,\,256,\,345,\,456\}$,\\
$\{013,\,015,\,134,\,256,\,345,\,356,\,456\}$, $\{013,\,016,\,024,\,025,\,026,\,045,\,246\}$,\\
$\{013,\,016,\,024,\,025,\,026,\,126,\,246\}$, $\{013,\,016,\,024,\,026,\,045,\,046,\,246\}$,\\
$\{013,\,016,\,025,\,026,\,035,\,045,\,256\}$, $\{013,\,016,\,025,\,026,\,045,\,246,\,256\}$,\\
$\{013,\,016,\,025,\,026,\,126,\,246,\,256\}$, $\{013,\,016,\,026,\,035,\,036,\,045,\,456\}$,\\
$\{013,\,016,\,026,\,035,\,036,\,356,\,456\}$, $\{013,\,016,\,026,\,035,\,045,\,256,\,456\}$,\\
$\{013,\,016,\,026,\,035,\,256,\,356,\,456\}$, $\{013,\,016,\,026,\,036,\,045,\,046,\,456\}$,\\
$\{013,\,016,\,026,\,036,\,046,\,136,\,456\}$, $\{013,\,016,\,026,\,036,\,136,\,356,\,456\}$,\\
$\{013,\,016,\,026,\,045,\,046,\,246,\,456\}$, $\{013,\,016,\,026,\,045,\,246,\,256,\,456\}$,\\
$\{013,\,016,\,026,\,046,\,136,\,246,\,456\}$, $\{013,\,016,\,026,\,126,\,136,\,246,\,256\}$,\\
$\{013,\,016,\,026,\,136,\,246,\,256,\,456\}$, $\{013,\,016,\,026,\,136,\,256,\,356,\,456\}$,\\
$\{013,\,016,\,126,\,136,\,156,\,246,\,256\}$, $\{013,\,016,\,136,\,156,\,246,\,256,\,456\}$,\\
$\{013,\,016,\,136,\,156,\,256,\,356,\,456\}$, $\{013,\,023,\,024,\,025,\,026,\,045,\,345\}$,\\
$\{013,\,023,\,024,\,025,\,026,\,123,\,234\}$, $\{013,\,023,\,024,\,025,\,026,\,234,\,345\}$,\\
$\{013,\,023,\,024,\,026,\,045,\,046,\,345\}$, $\{013,\,023,\,024,\,026,\,046,\,234,\,345\}$,\\
$\{013,\,023,\,025,\,026,\,035,\,045,\,345\}$, $\{013,\,023,\,025,\,026,\,035,\,235,\,345\}$,\\
$\{013,\,023,\,025,\,026,\,123,\,234,\,235\}$, $\{013,\,023,\,025,\,026,\,234,\,235,\,345\}$,\\
$\{013,\,023,\,026,\,035,\,036,\,045,\,345\}$, $\{013,\,023,\,026,\,035,\,036,\,235,\,345\}$,\\
$\{013,\,023,\,026,\,036,\,045,\,046,\,345\}$, $\{013,\,023,\,026,\,036,\,046,\,234,\,345\}$,\\
$\{013,\,023,\,026,\,036,\,123,\,234,\,235\}$, $\{013,\,023,\,026,\,036,\,234,\,235,\,345\}$,\\
$\{013,\,024,\,025,\,026,\,045,\,246,\,345\}$, $\{013,\,024,\,025,\,026,\,123,\,126,\,246\}$,\\
$\{013,\,024,\,025,\,026,\,123,\,234,\,246\}$, $\{013,\,024,\,025,\,026,\,234,\,246,\,345\}$,\\
$\{013,\,024,\,025,\,123,\,134,\,234,\,246\}$, $\{013,\,024,\,025,\,134,\,234,\,246,\,345\}$,\\
$\{013,\,024,\,026,\,045,\,046,\,246,\,345\}$, $\{013,\,024,\,026,\,046,\,234,\,246,\,345\}$,\\
$\{013,\,024,\,046,\,134,\,234,\,246,\,345\}$, $\{013,\,025,\,026,\,035,\,045,\,256,\,345\}$,\\
$\{013,\,025,\,026,\,035,\,235,\,256,\,345\}$, $\{013,\,025,\,026,\,045,\,246,\,256,\,345\}$,\\
$\{013,\,025,\,026,\,123,\,126,\,246,\,256\}$, $\{013,\,025,\,026,\,123,\,234,\,235,\,256\}$,\\
$\{013,\,025,\,026,\,123,\,234,\,246,\,256\}$, $\{013,\,025,\,026,\,234,\,235,\,256,\,345\}$,\\
$\{013,\,025,\,026,\,234,\,246,\,256,\,345\}$, $\{013,\,025,\,123,\,134,\,234,\,235,\,256\}$,\\
$\{013,\,025,\,123,\,134,\,234,\,246,\,256\}$, $\{013,\,025,\,134,\,234,\,235,\,256,\,345\}$,\\
$\{013,\,025,\,134,\,234,\,246,\,256,\,345\}$, $\{013,\,026,\,035,\,036,\,045,\,345,\,456\}$,\\
$\{013,\,026,\,035,\,036,\,235,\,345,\,356\}$, $\{013,\,026,\,035,\,036,\,345,\,356,\,456\}$,\\
$\{013,\,026,\,035,\,045,\,256,\,345,\,456\}$, $\{013,\,026,\,035,\,235,\,256,\,345,\,356\}$,\\
$\{013,\,026,\,035,\,256,\,345,\,356,\,456\}$, $\{013,\,026,\,036,\,045,\,046,\,345,\,456\}$,\\
$\{013,\,026,\,036,\,046,\,136,\,346,\,456\}$, $\{013,\,026,\,036,\,046,\,234,\,345,\,346\}$,\\
$\{013,\,026,\,036,\,046,\,345,\,346,\,456\}$, $\{013,\,026,\,036,\,123,\,136,\,346,\,356\}$,\\
$\{013,\,026,\,036,\,123,\,234,\,235,\,346\}$, $\{013,\,026,\,036,\,123,\,235,\,346,\,356\}$,\\
$\{013,\,026,\,036,\,136,\,346,\,356,\,456\}$, $\{013,\,026,\,036,\,234,\,235,\,345,\,346\}$,\\
$\{013,\,026,\,036,\,235,\,345,\,346,\,356\}$, $\{013,\,026,\,036,\,345,\,346,\,356,\,456\}$,\\
$\{013,\,026,\,045,\,046,\,246,\,345,\,456\}$, $\{013,\,026,\,045,\,246,\,256,\,345,\,456\}$,\\
$\{013,\,026,\,046,\,136,\,246,\,346,\,456\}$, $\{013,\,026,\,046,\,234,\,246,\,345,\,346\}$,\\
$\{013,\,026,\,046,\,246,\,345,\,346,\,456\}$, $\{013,\,026,\,123,\,126,\,136,\,246,\,256\}$,\\
$\{013,\,026,\,123,\,136,\,246,\,256,\,346\}$, $\{013,\,026,\,123,\,136,\,256,\,346,\,356\}$,\\
$\{013,\,026,\,123,\,234,\,235,\,256,\,346\}$, $\{013,\,026,\,123,\,234,\,246,\,256,\,346\}$,\\
$\{013,\,026,\,123,\,235,\,256,\,346,\,356\}$, $\{013,\,026,\,136,\,246,\,256,\,346,\,456\}$,\\
$\{013,\,026,\,136,\,256,\,346,\,356,\,456\}$, $\{013,\,026,\,234,\,235,\,256,\,345,\,346\}$,\\
$\{013,\,026,\,234,\,246,\,256,\,345,\,346\}$, $\{013,\,026,\,235,\,256,\,345,\,346,\,356\}$,\\
$\{013,\,026,\,246,\,256,\,345,\,346,\,456\}$, $\{013,\,026,\,256,\,345,\,346,\,356,\,456\}$,\\
$\{013,\,046,\,134,\,136,\,246,\,346,\,456\}$, $\{013,\,046,\,134,\,234,\,246,\,345,\,346\}$,\\
$\{013,\,046,\,134,\,246,\,345,\,346,\,456\}$, $\{013,\,123,\,126,\,136,\,156,\,246,\,256\}$,\\
$\{013,\,123,\,134,\,136,\,156,\,246,\,256\}$, $\{013,\,123,\,134,\,136,\,156,\,256,\,356\}$,\\
$\{013,\,123,\,134,\,136,\,246,\,256,\,346\}$, $\{013,\,123,\,134,\,136,\,256,\,346,\,356\}$,\\
$\{013,\,123,\,134,\,234,\,235,\,256,\,346\}$, $\{013,\,123,\,134,\,234,\,246,\,256,\,346\}$,\\
$\{013,\,123,\,134,\,235,\,256,\,346,\,356\}$, $\{013,\,134,\,136,\,156,\,246,\,256,\,456\}$,\\
$\{013,\,134,\,136,\,156,\,256,\,356,\,456\}$, $\{013,\,134,\,136,\,246,\,256,\,346,\,456\}$,\\
$\{013,\,134,\,136,\,256,\,346,\,356,\,456\}$, $\{013,\,134,\,234,\,235,\,256,\,345,\,346\}$,\\
$\{013,\,134,\,234,\,246,\,256,\,345,\,346\}$, $\{013,\,134,\,235,\,256,\,345,\,346,\,356\}$,\\
$\{013,\,134,\,246,\,256,\,345,\,346,\,456\}$, $\{013,\,134,\,256,\,345,\,346,\,356,\,456\}$,\\
$\{014,\,015,\,024,\,045,\,124,\,246,\,345\}$, $\{014,\,015,\,024,\,123,\,124,\,126,\,246\}$,\\
$\{014,\,015,\,024,\,123,\,124,\,134,\,246\}$, $\{014,\,015,\,024,\,124,\,134,\,246,\,345\}$,\\
$\{014,\,015,\,045,\,124,\,145,\,246,\,345\}$, $\{014,\,015,\,045,\,145,\,246,\,345,\,456\}$,\\
$\{014,\,015,\,123,\,124,\,126,\,156,\,246\}$, $\{014,\,015,\,123,\,124,\,134,\,156,\,246\}$,\\
$\{014,\,015,\,124,\,134,\,145,\,156,\,246\}$, $\{014,\,015,\,124,\,134,\,145,\,246,\,345\}$,\\
$\{014,\,015,\,134,\,145,\,156,\,246,\,456\}$, $\{014,\,015,\,134,\,145,\,246,\,345,\,456\}$,\\
$\{014,\,123,\,124,\,126,\,136,\,156,\,246\}$, $\{014,\,123,\,124,\,134,\,136,\,156,\,246\}$,\\
$\{015,\,024,\,025,\,045,\,124,\,246,\,345\}$, $\{015,\,024,\,025,\,123,\,124,\,126,\,246\}$,\\
$\{015,\,024,\,025,\,123,\,124,\,134,\,246\}$, $\{015,\,024,\,025,\,124,\,134,\,246,\,345\}$,\\
$\{015,\,025,\,045,\,124,\,125,\,246,\,345\}$, $\{015,\,025,\,045,\,125,\,246,\,256,\,345\}$,\\
$\{015,\,025,\,123,\,124,\,125,\,126,\,246\}$, $\{015,\,025,\,123,\,124,\,125,\,134,\,246\}$,\\
$\{015,\,025,\,123,\,125,\,126,\,246,\,256\}$, $\{015,\,025,\,123,\,125,\,134,\,235,\,256\}$,\\
$\{015,\,025,\,123,\,125,\,134,\,246,\,256\}$, $\{015,\,025,\,124,\,125,\,134,\,246,\,345\}$,\\
$\{015,\,025,\,125,\,134,\,235,\,256,\,345\}$, $\{015,\,025,\,125,\,134,\,246,\,256,\,345\}$,\\
$\{015,\,045,\,124,\,125,\,145,\,246,\,345\}$, $\{015,\,045,\,125,\,145,\,246,\,256,\,345\}$,\\
$\{015,\,045,\,145,\,246,\,256,\,345,\,456\}$, $\{015,\,123,\,124,\,125,\,126,\,156,\,246\}$,\\
$\{015,\,123,\,124,\,125,\,134,\,156,\,246\}$, $\{015,\,123,\,125,\,126,\,156,\,246,\,256\}$,\\
$\{015,\,123,\,125,\,134,\,156,\,246,\,256\}$, $\{015,\,123,\,125,\,134,\,156,\,256,\,356\}$,\\
$\{015,\,123,\,125,\,134,\,235,\,256,\,356\}$, $\{015,\,124,\,125,\,134,\,145,\,156,\,246\}$,\\
$\{015,\,124,\,125,\,134,\,145,\,246,\,345\}$, $\{015,\,125,\,134,\,145,\,156,\,246,\,256\}$,\\
$\{015,\,125,\,134,\,145,\,156,\,256,\,356\}$, $\{015,\,125,\,134,\,145,\,246,\,256,\,345\}$,\\
$\{015,\,125,\,134,\,145,\,256,\,345,\,356\}$, $\{015,\,125,\,134,\,235,\,256,\,345,\,356\}$,\\
$\{015,\,134,\,145,\,156,\,246,\,256,\,456\}$, $\{015,\,134,\,145,\,156,\,256,\,356,\,456\}$,\\
$\{015,\,134,\,145,\,246,\,256,\,345,\,456\}$, $\{015,\,134,\,145,\,256,\,345,\,356,\,456\}$,\\
$\{024,\,025,\,123,\,124,\,134,\,234,\,246\}$, $\{024,\,025,\,124,\,134,\,234,\,246,\,345\}$,\\
$\{025,\,123,\,124,\,125,\,134,\,234,\,246\}$, $\{025,\,123,\,125,\,134,\,234,\,235,\,256\}$,\\
$\{025,\,123,\,125,\,134,\,234,\,246,\,256\}$, $\{025,\,124,\,125,\,134,\,234,\,246,\,345\}$,\\
$\{025,\,125,\,134,\,234,\,235,\,256,\,345\}$, $\{025,\,125,\,134,\,234,\,246,\,256,\,345\}$ $\}$
\end{center}

\subsection*{Edges of \texorpdfstring{$\Delta$}{Delta} (238)}
{\footnotesize\sloppy $013 \mid 014$,\ $013 \mid 015$,\ $013 \mid 016$,\ $013 \mid 023$,\ $013 \mid 024$,\ $013 \mid 025$,\ $013 \mid 026$,\ $013 \mid 035$,\ $013 \mid 036$,\ $013 \mid 045$,\ $013 \mid 046$,\ $013 \mid 123$,\ $013 \mid 126$,\ $013 \mid 134$,\ $013 \mid 136$,\ $013 \mid 156$,\ $013 \mid 234$,\ $013 \mid 235$,\ $013 \mid 246$,\ $013 \mid 256$,\ $013 \mid 345$,\ $013 \mid 346$,\ $013 \mid 356$,\ $013 \mid 456$,\ $014 \mid 015$,\ $014 \mid 016$,\ $014 \mid 024$,\ $014 \mid 045$,\ $014 \mid 046$,\ $014 \mid 123$,\ $014 \mid 124$,\ $014 \mid 126$,\ $014 \mid 134$,\ $014 \mid 136$,\ $014 \mid 145$,\ $014 \mid 156$,\ $014 \mid 246$,\ $014 \mid 345$,\ $014 \mid 456$,\ $015 \mid 016$,\ $015 \mid 024$,\ $015 \mid 025$,\ $015 \mid 035$,\ $015 \mid 045$,\ $015 \mid 123$,\ $015 \mid 124$,\ $015 \mid 125$,\ $015 \mid 126$,\ $015 \mid 134$,\ $015 \mid 145$,\ $015 \mid 156$,\ $015 \mid 235$,\ $015 \mid 246$,\ $015 \mid 256$,\ $015 \mid 345$,\ $015 \mid 356$,\ $015 \mid 456$,\ $016 \mid 024$,\ $016 \mid 025$,\ $016 \mid 026$,\ $016 \mid 035$,\ $016 \mid 036$,\ $016 \mid 045$,\ $016 \mid 046$,\ $016 \mid 126$,\ $016 \mid 136$,\ $016 \mid 156$,\ $016 \mid 246$,\ $016 \mid 256$,\ $016 \mid 356$,\ $016 \mid 456$,\ $023 \mid 024$,\ $023 \mid 025$,\ $023 \mid 026$,\ $023 \mid 035$,\ $023 \mid 036$,\ $023 \mid 045$,\ $023 \mid 046$,\ $023 \mid 123$,\ $023 \mid 234$,\ $023 \mid 235$,\ $023 \mid 345$,\ $024 \mid 025$,\ $024 \mid 026$,\ $024 \mid 045$,\ $024 \mid 046$,\ $024 \mid 123$,\ $024 \mid 124$,\ $024 \mid 126$,\ $024 \mid 134$,\ $024 \mid 234$,\ $024 \mid 246$,\ $024 \mid 345$,\ $025 \mid 026$,\ $025 \mid 035$,\ $025 \mid 045$,\ $025 \mid 123$,\ $025 \mid 124$,\ $025 \mid 125$,\ $025 \mid 126$,\ $025 \mid 134$,\ $025 \mid 234$,\ $025 \mid 235$,\ $025 \mid 246$,\ $025 \mid 256$,\ $025 \mid 345$,\ $026 \mid 035$,\ $026 \mid 036$,\ $026 \mid 045$,\ $026 \mid 046$,\ $026 \mid 123$,\ $026 \mid 126$,\ $026 \mid 136$,\ $026 \mid 234$,\ $026 \mid 235$,\ $026 \mid 246$,\ $026 \mid 256$,\ $026 \mid 345$,\ $026 \mid 346$,\ $026 \mid 356$,\ $026 \mid 456$,\ $035 \mid 036$,\ $035 \mid 045$,\ $035 \mid 235$,\ $035 \mid 256$,\ $035 \mid 345$,\ $035 \mid 356$,\ $035 \mid 456$,\ $036 \mid 045$,\ $036 \mid 046$,\ $036 \mid 123$,\ $036 \mid 136$,\ $036 \mid 234$,\ $036 \mid 235$,\ $036 \mid 345$,\ $036 \mid 346$,\ $036 \mid 356$,\ $036 \mid 456$,\ $045 \mid 046$,\ $045 \mid 124$,\ $045 \mid 125$,\ $045 \mid 145$,\ $045 \mid 246$,\ $045 \mid 256$,\ $045 \mid 345$,\ $045 \mid 456$,\ $046 \mid 134$,\ $046 \mid 136$,\ $046 \mid 234$,\ $046 \mid 246$,\ $046 \mid 345$,\ $046 \mid 346$,\ $046 \mid 456$,\ $123 \mid 124$,\ $123 \mid 125$,\ $123 \mid 126$,\ $123 \mid 134$,\ $123 \mid 136$,\ $123 \mid 156$,\ $123 \mid 234$,\ $123 \mid 235$,\ $123 \mid 246$,\ $123 \mid 256$,\ $123 \mid 346$,\ $123 \mid 356$,\ $124 \mid 125$,\ $124 \mid 126$,\ $124 \mid 134$,\ $124 \mid 136$,\ $124 \mid 145$,\ $124 \mid 156$,\ $124 \mid 234$,\ $124 \mid 246$,\ $124 \mid 345$,\ $125 \mid 126$,\ $125 \mid 134$,\ $125 \mid 145$,\ $125 \mid 156$,\ $125 \mid 234$,\ $125 \mid 235$,\ $125 \mid 246$,\ $125 \mid 256$,\ $125 \mid 345$,\ $125 \mid 356$,\ $126 \mid 136$,\ $126 \mid 156$,\ $126 \mid 246$,\ $126 \mid 256$,\ $134 \mid 136$,\ $134 \mid 145$,\ $134 \mid 156$,\ $134 \mid 234$,\ $134 \mid 235$,\ $134 \mid 246$,\ $134 \mid 256$,\ $134 \mid 345$,\ $134 \mid 346$,\ $134 \mid 356$,\ $134 \mid 456$,\ $136 \mid 156$,\ $136 \mid 246$,\ $136 \mid 256$,\ $136 \mid 346$,\ $136 \mid 356$,\ $136 \mid 456$,\ $145 \mid 156$,\ $145 \mid 246$,\ $145 \mid 256$,\ $145 \mid 345$,\ $145 \mid 356$,\ $145 \mid 456$,\ $156 \mid 246$,\ $156 \mid 256$,\ $156 \mid 356$,\ $156 \mid 456$,\ $234 \mid 235$,\ $234 \mid 246$,\ $234 \mid 256$,\ $234 \mid 345$,\ $234 \mid 346$,\ $235 \mid 256$,\ $235 \mid 345$,\ $235 \mid 346$,\ $235 \mid 356$,\ $246 \mid 256$,\ $246 \mid 345$,\ $246 \mid 346$,\ $246 \mid 456$,\ $256 \mid 345$,\ $256 \mid 346$,\ $256 \mid 356$,\ $256 \mid 456$,\ $345 \mid 346$,\ $345 \mid 356$,\ $345 \mid 456$,\ $346 \mid 356$,\ $346 \mid 456$,\ $356 \mid 456$.\par}

\subsection*{Non-faces of \texorpdfstring{$\Delta$}{Delta} (140 non-edges)}
{\footnotesize\sloppy $013 \mid 124$,\ $013 \mid 125$,\ $013 \mid 145$,\ $014 \mid 023$,\ $014 \mid 025$,\ $014 \mid 026$,\ $014 \mid 035$,\ $014 \mid 036$,\ $014 \mid 125$,\ $014 \mid 234$,\ $014 \mid 235$,\ $014 \mid 256$,\ $014 \mid 346$,\ $014 \mid 356$,\ $015 \mid 023$,\ $015 \mid 026$,\ $015 \mid 036$,\ $015 \mid 046$,\ $015 \mid 136$,\ $015 \mid 234$,\ $015 \mid 346$,\ $016 \mid 023$,\ $016 \mid 123$,\ $016 \mid 124$,\ $016 \mid 125$,\ $016 \mid 134$,\ $016 \mid 145$,\ $016 \mid 234$,\ $016 \mid 235$,\ $016 \mid 345$,\ $016 \mid 346$,\ $023 \mid 124$,\ $023 \mid 125$,\ $023 \mid 126$,\ $023 \mid 134$,\ $023 \mid 136$,\ $023 \mid 145$,\ $023 \mid 156$,\ $023 \mid 246$,\ $023 \mid 256$,\ $023 \mid 346$,\ $023 \mid 356$,\ $023 \mid 456$,\ $024 \mid 035$,\ $024 \mid 036$,\ $024 \mid 125$,\ $024 \mid 136$,\ $024 \mid 145$,\ $024 \mid 156$,\ $024 \mid 235$,\ $024 \mid 256$,\ $024 \mid 346$,\ $024 \mid 356$,\ $024 \mid 456$,\ $025 \mid 036$,\ $025 \mid 046$,\ $025 \mid 136$,\ $025 \mid 145$,\ $025 \mid 156$,\ $025 \mid 346$,\ $025 \mid 356$,\ $025 \mid 456$,\ $026 \mid 124$,\ $026 \mid 125$,\ $026 \mid 134$,\ $026 \mid 145$,\ $026 \mid 156$,\ $035 \mid 046$,\ $035 \mid 123$,\ $035 \mid 124$,\ $035 \mid 125$,\ $035 \mid 126$,\ $035 \mid 134$,\ $035 \mid 136$,\ $035 \mid 145$,\ $035 \mid 156$,\ $035 \mid 234$,\ $035 \mid 246$,\ $035 \mid 346$,\ $036 \mid 124$,\ $036 \mid 125$,\ $036 \mid 126$,\ $036 \mid 134$,\ $036 \mid 145$,\ $036 \mid 156$,\ $036 \mid 246$,\ $036 \mid 256$,\ $045 \mid 123$,\ $045 \mid 126$,\ $045 \mid 134$,\ $045 \mid 136$,\ $045 \mid 156$,\ $045 \mid 234$,\ $045 \mid 235$,\ $045 \mid 346$,\ $045 \mid 356$,\ $046 \mid 123$,\ $046 \mid 124$,\ $046 \mid 125$,\ $046 \mid 126$,\ $046 \mid 145$,\ $046 \mid 156$,\ $046 \mid 235$,\ $046 \mid 256$,\ $046 \mid 356$,\ $123 \mid 145$,\ $123 \mid 345$,\ $123 \mid 456$,\ $124 \mid 235$,\ $124 \mid 256$,\ $124 \mid 346$,\ $124 \mid 356$,\ $124 \mid 456$,\ $125 \mid 136$,\ $125 \mid 346$,\ $125 \mid 456$,\ $126 \mid 134$,\ $126 \mid 145$,\ $126 \mid 234$,\ $126 \mid 235$,\ $126 \mid 345$,\ $126 \mid 346$,\ $126 \mid 356$,\ $126 \mid 456$,\ $136 \mid 145$,\ $136 \mid 234$,\ $136 \mid 235$,\ $136 \mid 345$,\ $145 \mid 234$,\ $145 \mid 235$,\ $145 \mid 346$,\ $156 \mid 234$,\ $156 \mid 235$,\ $156 \mid 345$,\ $156 \mid 346$,\ $234 \mid 356$,\ $234 \mid 456$,\ $235 \mid 246$,\ $235 \mid 456$,\ $246 \mid 356$.\par}

\subsection*{The 126 singleton fibers (forced edges)}
{\footnotesize\sloppy $013 \mid 014$,\ $013 \mid 015$,\ $013 \mid 016$,\ $013 \mid 023$,\ $013 \mid 035$,\ $013 \mid 036$,\ $013 \mid 123$,\ $013 \mid 134$,\ $013 \mid 136$,\ $014 \mid 015$,\ $014 \mid 016$,\ $014 \mid 024$,\ $014 \mid 045$,\ $014 \mid 046$,\ $014 \mid 124$,\ $014 \mid 134$,\ $014 \mid 145$,\ $015 \mid 016$,\ $015 \mid 025$,\ $015 \mid 035$,\ $015 \mid 045$,\ $015 \mid 125$,\ $015 \mid 145$,\ $015 \mid 156$,\ $016 \mid 026$,\ $016 \mid 036$,\ $016 \mid 046$,\ $016 \mid 126$,\ $016 \mid 136$,\ $016 \mid 156$,\ $023 \mid 024$,\ $023 \mid 025$,\ $023 \mid 026$,\ $023 \mid 035$,\ $023 \mid 036$,\ $023 \mid 123$,\ $023 \mid 234$,\ $023 \mid 235$,\ $024 \mid 025$,\ $024 \mid 026$,\ $024 \mid 045$,\ $024 \mid 046$,\ $024 \mid 124$,\ $024 \mid 234$,\ $024 \mid 246$,\ $025 \mid 026$,\ $025 \mid 035$,\ $025 \mid 045$,\ $025 \mid 125$,\ $025 \mid 235$,\ $025 \mid 256$,\ $026 \mid 036$,\ $026 \mid 046$,\ $026 \mid 126$,\ $026 \mid 246$,\ $026 \mid 256$,\ $035 \mid 036$,\ $035 \mid 045$,\ $035 \mid 235$,\ $035 \mid 345$,\ $035 \mid 356$,\ $036 \mid 046$,\ $036 \mid 136$,\ $036 \mid 346$,\ $036 \mid 356$,\ $045 \mid 046$,\ $045 \mid 145$,\ $045 \mid 345$,\ $045 \mid 456$,\ $046 \mid 246$,\ $046 \mid 346$,\ $046 \mid 456$,\ $123 \mid 124$,\ $123 \mid 125$,\ $123 \mid 126$,\ $123 \mid 134$,\ $123 \mid 136$,\ $123 \mid 234$,\ $123 \mid 235$,\ $124 \mid 125$,\ $124 \mid 126$,\ $124 \mid 134$,\ $124 \mid 145$,\ $124 \mid 234$,\ $124 \mid 246$,\ $125 \mid 126$,\ $125 \mid 145$,\ $125 \mid 156$,\ $125 \mid 235$,\ $125 \mid 256$,\ $126 \mid 136$,\ $126 \mid 156$,\ $126 \mid 246$,\ $126 \mid 256$,\ $134 \mid 136$,\ $134 \mid 145$,\ $134 \mid 234$,\ $134 \mid 345$,\ $134 \mid 346$,\ $136 \mid 156$,\ $136 \mid 346$,\ $136 \mid 356$,\ $145 \mid 156$,\ $145 \mid 345$,\ $145 \mid 456$,\ $156 \mid 256$,\ $156 \mid 356$,\ $156 \mid 456$,\ $234 \mid 235$,\ $234 \mid 246$,\ $234 \mid 345$,\ $234 \mid 346$,\ $235 \mid 256$,\ $235 \mid 345$,\ $235 \mid 356$,\ $246 \mid 256$,\ $246 \mid 346$,\ $246 \mid 456$,\ $256 \mid 356$,\ $256 \mid 456$,\ $345 \mid 346$,\ $345 \mid 356$,\ $345 \mid 456$,\ $346 \mid 356$,\ $346 \mid 456$,\ $356 \mid 456$.\par}

\subsection*{The 105 two-way fibers}
Each fiber offers two candidate edges; the one in $\Delta$ is underlined.
{\footnotesize\sloppy $\{\underline{013 \mid 024},\ 014 \mid 023\}$,\ $\{\underline{013 \mid 025},\ 015 \mid 023\}$,\ $\{\underline{013 \mid 026},\ 016 \mid 023\}$,\ $\{\underline{013 \mid 045},\ 014 \mid 035\}$,\ $\{\underline{013 \mid 046},\ 014 \mid 036\}$,\ $\{013 \mid 124,\ \underline{014 \mid 123}\}$,\ $\{013 \mid 125,\ \underline{015 \mid 123}\}$,\ $\{\underline{013 \mid 126},\ 016 \mid 123\}$,\ $\{013 \mid 145,\ \underline{015 \mid 134}\}$,\ $\{\underline{013 \mid 156},\ 015 \mid 136\}$,\ $\{\underline{013 \mid 234},\ 023 \mid 134\}$,\ $\{\underline{013 \mid 235},\ 035 \mid 123\}$,\ $\{\underline{013 \mid 345},\ 035 \mid 134\}$,\ $\{\underline{013 \mid 346},\ 036 \mid 134\}$,\ $\{\underline{013 \mid 356},\ 035 \mid 136\}$,\ $\{014 \mid 025,\ \underline{015 \mid 024}\}$,\ $\{014 \mid 026,\ \underline{016 \mid 024}\}$,\ $\{014 \mid 125,\ \underline{015 \mid 124}\}$,\ $\{\underline{014 \mid 126},\ 016 \mid 124\}$,\ $\{\underline{014 \mid 136},\ 016 \mid 134\}$,\ $\{\underline{014 \mid 156},\ 016 \mid 145\}$,\ $\{014 \mid 234,\ \underline{024 \mid 134}\}$,\ $\{\underline{014 \mid 246},\ 046 \mid 124\}$,\ $\{\underline{014 \mid 345},\ 045 \mid 134\}$,\ $\{014 \mid 346,\ \underline{046 \mid 134}\}$,\ $\{\underline{014 \mid 456},\ 046 \mid 145\}$,\ $\{015 \mid 026,\ \underline{016 \mid 025}\}$,\ $\{015 \mid 036,\ \underline{016 \mid 035}\}$,\ $\{015 \mid 046,\ \underline{016 \mid 045}\}$,\ $\{\underline{015 \mid 126},\ 016 \mid 125\}$,\ $\{\underline{015 \mid 235},\ 035 \mid 125\}$,\ $\{\underline{015 \mid 256},\ 025 \mid 156\}$,\ $\{\underline{015 \mid 345},\ 035 \mid 145\}$,\ $\{\underline{015 \mid 356},\ 035 \mid 156\}$,\ $\{\underline{015 \mid 456},\ 045 \mid 156\}$,\ $\{\underline{016 \mid 246},\ 046 \mid 126\}$,\ $\{\underline{016 \mid 256},\ 026 \mid 156\}$,\ $\{016 \mid 346,\ \underline{046 \mid 136}\}$,\ $\{\underline{016 \mid 356},\ 036 \mid 156\}$,\ $\{\underline{016 \mid 456},\ 046 \mid 156\}$,\ $\{\underline{023 \mid 045},\ 024 \mid 035\}$,\ $\{\underline{023 \mid 046},\ 024 \mid 036\}$,\ $\{023 \mid 124,\ \underline{024 \mid 123}\}$,\ $\{023 \mid 125,\ \underline{025 \mid 123}\}$,\ $\{023 \mid 126,\ \underline{026 \mid 123}\}$,\ $\{023 \mid 136,\ \underline{036 \mid 123}\}$,\ $\{023 \mid 246,\ \underline{026 \mid 234}\}$,\ $\{023 \mid 256,\ \underline{026 \mid 235}\}$,\ $\{\underline{023 \mid 345},\ 035 \mid 234\}$,\ $\{023 \mid 346,\ \underline{036 \mid 234}\}$,\ $\{023 \mid 356,\ \underline{036 \mid 235}\}$,\ $\{024 \mid 125,\ \underline{025 \mid 124}\}$,\ $\{\underline{024 \mid 126},\ 026 \mid 124\}$,\ $\{024 \mid 145,\ \underline{045 \mid 124}\}$,\ $\{024 \mid 235,\ \underline{025 \mid 234}\}$,\ $\{024 \mid 256,\ \underline{025 \mid 246}\}$,\ $\{\underline{024 \mid 345},\ 045 \mid 234\}$,\ $\{024 \mid 346,\ \underline{046 \mid 234}\}$,\ $\{024 \mid 456,\ \underline{045 \mid 246}\}$,\ $\{025 \mid 036,\ \underline{026 \mid 035}\}$,\ $\{025 \mid 046,\ \underline{026 \mid 045}\}$,\ $\{\underline{025 \mid 126},\ 026 \mid 125\}$,\ $\{025 \mid 145,\ \underline{045 \mid 125}\}$,\ $\{\underline{025 \mid 345},\ 045 \mid 235\}$,\ $\{025 \mid 356,\ \underline{035 \mid 256}\}$,\ $\{025 \mid 456,\ \underline{045 \mid 256}\}$,\ $\{\underline{026 \mid 136},\ 036 \mid 126\}$,\ $\{\underline{026 \mid 346},\ 036 \mid 246\}$,\ $\{\underline{026 \mid 356},\ 036 \mid 256\}$,\ $\{\underline{026 \mid 456},\ 046 \mid 256\}$,\ $\{035 \mid 046,\ \underline{036 \mid 045}\}$,\ $\{035 \mid 346,\ \underline{036 \mid 345}\}$,\ $\{\underline{035 \mid 456},\ 045 \mid 356\}$,\ $\{\underline{036 \mid 456},\ 046 \mid 356\}$,\ $\{045 \mid 346,\ \underline{046 \mid 345}\}$,\ $\{123 \mid 145,\ \underline{125 \mid 134}\}$,\ $\{\underline{123 \mid 156},\ 125 \mid 136\}$,\ $\{\underline{123 \mid 246},\ 126 \mid 234\}$,\ $\{\underline{123 \mid 256},\ 126 \mid 235\}$,\ $\{123 \mid 345,\ \underline{134 \mid 235}\}$,\ $\{\underline{123 \mid 346},\ 136 \mid 234\}$,\ $\{\underline{123 \mid 356},\ 136 \mid 235\}$,\ $\{\underline{124 \mid 136},\ 126 \mid 134\}$,\ $\{\underline{124 \mid 156},\ 126 \mid 145\}$,\ $\{124 \mid 235,\ \underline{125 \mid 234}\}$,\ $\{124 \mid 256,\ \underline{125 \mid 246}\}$,\ $\{\underline{124 \mid 345},\ 145 \mid 234\}$,\ $\{124 \mid 346,\ \underline{134 \mid 246}\}$,\ $\{124 \mid 456,\ \underline{145 \mid 246}\}$,\ $\{\underline{125 \mid 345},\ 145 \mid 235\}$,\ $\{\underline{125 \mid 356},\ 156 \mid 235\}$,\ $\{125 \mid 456,\ \underline{145 \mid 256}\}$,\ $\{126 \mid 346,\ \underline{136 \mid 246}\}$,\ $\{126 \mid 356,\ \underline{136 \mid 256}\}$,\ $\{126 \mid 456,\ \underline{156 \mid 246}\}$,\ $\{\underline{134 \mid 156},\ 136 \mid 145\}$,\ $\{\underline{134 \mid 356},\ 136 \mid 345\}$,\ $\{\underline{134 \mid 456},\ 145 \mid 346\}$,\ $\{\underline{136 \mid 456},\ 156 \mid 346\}$,\ $\{\underline{145 \mid 356},\ 156 \mid 345\}$,\ $\{\underline{234 \mid 256},\ 235 \mid 246\}$,\ $\{234 \mid 356,\ \underline{235 \mid 346}\}$,\ $\{234 \mid 456,\ \underline{246 \mid 345}\}$,\ $\{235 \mid 456,\ \underline{256 \mid 345}\}$,\ $\{246 \mid 356,\ \underline{256 \mid 346}\}$.\par}

\subsection*{The seven fibers of size six}
One fiber for each ground-set element; each offers six candidate edges, the one in $\Delta$ underlined.
\begin{itemize}\footnotesize
\item $\{\underline{013 \mid 246},\ 016 \mid 234,\ 024 \mid 136,\ 026 \mid 134,\ 036 \mid 124,\ 046 \mid 123\}$
\item $\{\underline{013 \mid 256},\ 016 \mid 235,\ 023 \mid 156,\ 025 \mid 136,\ 035 \mid 126,\ 036 \mid 125\}$
\item $\{\underline{013 \mid 456},\ 014 \mid 356,\ 015 \mid 346,\ 016 \mid 345,\ 036 \mid 145,\ 045 \mid 136\}$
\item $\{014 \mid 235,\ 015 \mid 234,\ 023 \mid 145,\ \underline{025 \mid 134},\ 035 \mid 124,\ 045 \mid 123\}$
\item $\{014 \mid 256,\ \underline{015 \mid 246},\ 024 \mid 156,\ 026 \mid 145,\ 045 \mid 126,\ 046 \mid 125\}$
\item $\{023 \mid 456,\ 024 \mid 356,\ 025 \mid 346,\ \underline{026 \mid 345},\ 035 \mid 246,\ 046 \mid 235\}$
\item $\{123 \mid 456,\ 124 \mid 356,\ 125 \mid 346,\ 126 \mid 345,\ \underline{134 \mid 256},\ 156 \mid 234\}$
\end{itemize}

\end{document}